\documentclass[12pt,a4paper]{amsart}
\pdfoutput=1

\usepackage{a4wide}
\usepackage{amsthm, amsmath, amssymb}

\usepackage[round]{natbib}

\usepackage[plain,noend]{algorithm2e}

\usepackage{graphicx, color, enumerate, bm}
\usepackage{xr-hyper}
\usepackage[colorlinks,citecolor=blue,urlcolor=blue]{hyperref}

\usepackage{caption}
\usepackage{subcaption}
\usepackage[makeroom]{cancel}
\usepackage{booktabs}

\newtheorem{theorem}{Theorem}[section]
\newtheorem{lemma}[theorem]{Lemma}
\newtheorem{proposition}[theorem]{Proposition}
\newtheorem{property}{Property}[section]

\theoremstyle{definition}
\newtheorem{definition}[theorem]{Definition}
\newtheorem{example}[theorem]{Example}
\newtheorem{remark}[theorem]{Remark}

\numberwithin{equation}{section}

\input{commands}
\newenvironment{proofof}[1][]{\begin{trivlist}
  \item[\hskip \labelsep {\bfseries Proof of #1.}]}{\qed\end{trivlist}}

\title[Symmetric Rank Covariance]{Symmetric Rank Covariances: a Generalised Framework for Nonparametric Measures of Dependence}

\author{Luca Weihs}
\address{Department of Statistics, University
  of Washington, Seattle, WA, U.S.A.}
\email{lucaw@uw.edu}

\author{Mathias Drton}
\address{Department of Statistics, University
  of Washington, Seattle, WA, U.S.A.}
\email{md5@uw.edu}

\author{Nicolai Meinshausen}
\address{Seminar for Statistics, ETH Z{\"u}rich, Switzerland}
\email{meinshausen@stat.math.ethz.ch}


\begin{document}
\allowdisplaybreaks

\begin{abstract}
  The need to test whether two random vectors are independent has
  spawned a large number of competing measures of dependence.  We
  are interested in nonparametric measures that are invariant under strictly
  increasing transformations, such as Kendall's tau, Hoeffding's D,
  and the more recently discovered Bergsma--Dassios sign
  covariance. Each of these measures exhibits symmetries that are not
  readily apparent from their definitions.  Making these symmetries
  explicit, we define a new class of multivariate nonparametric
  measures of dependence that we refer to as Symmetric Rank
  Covariances. This new class generalises all of the above measures
  and leads naturally to multivariate extensions of the
  Bergsma--Dassios sign covariance. Symmetric Rank Covariances may be
  estimated unbiasedly using U-statistics for which we prove results
  on computational efficiency and large-sample behavior.  The
  algorithms we develop for their computation include, to the best of
  our knowledge, the first efficient algorithms for the well-known
  Hoeffding's D statistic in the multivariate setting.
\end{abstract}

\maketitle

\section{Introduction}\label{section: introduction}

Many applications, from gene expression analysis to feature selection in machine learning tasks, require quantifying the dependence between collections of random variables. Letting $X=(X_1,...,X_r)$ and $Y = (Y_1,...,Y_s)$ be random vectors, we are interested in measures of dependence $\mu$ which exhibit the following three properties,

\begin{enumerate}[(1)]
\item \emph{I-consistency}: if $X,\ Y$ are independent then $\mu(X,Y) = 0$,
\item \emph{D-consistency}: if $X,\ Y$ are dependent then $\mu(X,Y) \not= 0$,
\item \emph{Monotonic invariance}: if $f_1,\dots,f_r,g_1,\dots,g_s$ are
  strictly increasing functions then
  $\mu(X,Y) = \mu((f_1(X_1),\dots,f_n(X_r)), (g_1(Y_1),\dots,g_m(Y_s))$.
  For simpler language, we also refer to this property as $\mu$ being
  \emph{nonparametric}.
\end{enumerate}

If $\mu$ is I-consistent then tests of independence can be based on
the null hypothesis $\mu(X,Y)=0$.  If $\mu$ is additionally
D-consistent then tests based on consistent estimators of $\mu$ are
guarenteed to asymptotically reject independence when it fails to
hold. When $\mu$ is both I- and D-consistent we will simply call it
\emph{consistent}. On the other hand, monotonic invariance is the
intuitive requirement that the level of dependence between two random
vectors is invariant to monotonic transformations of any
coordinate. Unfortunately, many popular measures of dependence fail to
satisfy some subset of these properties. For instance, Kendall's
$\tau$ \citep{Kendall38} and Spearman's $\rho$ \citep{Spearman04} are
nonparametric and I-consistent but not D-consistent while the distance
correlation \citep{SzekelyEtAl07} is consistent but not nonparametric
in the above sense.

For bivariate observations, \cite{Hoeffding48} introduced a
nonparametric dependence measure that is consistent for a large class
of continuous distributions.  Let $(X,Y)$ be a random vector taking
values in $\bR^2$, with joint and marginal distribution
functions $F_{XY}$, $F_X$, and $F_Y$.  Then the statistic, now called
Hoeffding's $D$, is defined as
\begin{align}
  \label{eq:hoeffding-D-R2}
  D=\int_{\bR^2} (F_{XY}(x,y) - F_X(x)F_Y(y))^2\dF_{XY}(x,y).
\end{align}
\cite{BergsmaDassios14} introduced a new bivariate dependence measure
$\tau^*$ that is nonparametric and improves upon Hoeffding's $D$ by
guaranteeing consistency for all bivariate mixtures of continuous and
discrete distributions.  As its name suggests, $\tau^*$ generalises
Kendall's $\tau$; where $\tau$ counts concordant and discordant pairs
of points, $\tau^*$ counts concordant and discordant quadruples of
points.  The proof of consistency of $\tau^*$ is considerably more
involved than that for $D$.

Surprisingly both $D$ and $\tau^*$ exhibit a number of identical
symmetries that are obfuscated by their usual definitions. Indeed, as will
be made precise, $D$ and $\tau^*$ can be represented as the covariance
between signed sums of indicator functions acted on by the subgroup
\begin{align*}
	H = \langle (1\ 4),\ (2\ 3) \rangle
\end{align*}
of the symmetric group on four elements. We generalise the above
observation to define a new class of dependence measures called
Symmetric Rank Covariances. All such measures are I-consistent,
nonparametric, and include $D$, $\tau^*$, $\tau$, and $\rho$ as
special cases. Moreover, our new class of measures includes natural
multivariate extensions of $\tau^*$ which themselves inspire new notions
of concordance and discordance in higher dimensions, see Figure 
\ref{fig:tau-p-concord-discord}. While Symmetric Rank Covariances
need not always be D-consistent we identify a sub-collection of
measures that are.  These consistent measures can be interpreted as
testing independence by applying, possibly infinitely many,
independence tests to discretizations of $(X,Y)$. Symmetric Rank
Covariances can be readily estimated using U-statistics and we show
that the use of efficient data structures for orthogonal range queries
can result in substantial savings. Moreover, we show
that under independence many of the resulting U-statistics are
degenerate of order 2, thus having non-Gaussian limiting distributions. For
space, most proofs have been moved to Appendix \ref{app:proofs}.

\begin{figure}
  \centering
  \begin{subfigure}[t]{0.66\textwidth}
    \centering
    \includegraphics[width=.4\textwidth]{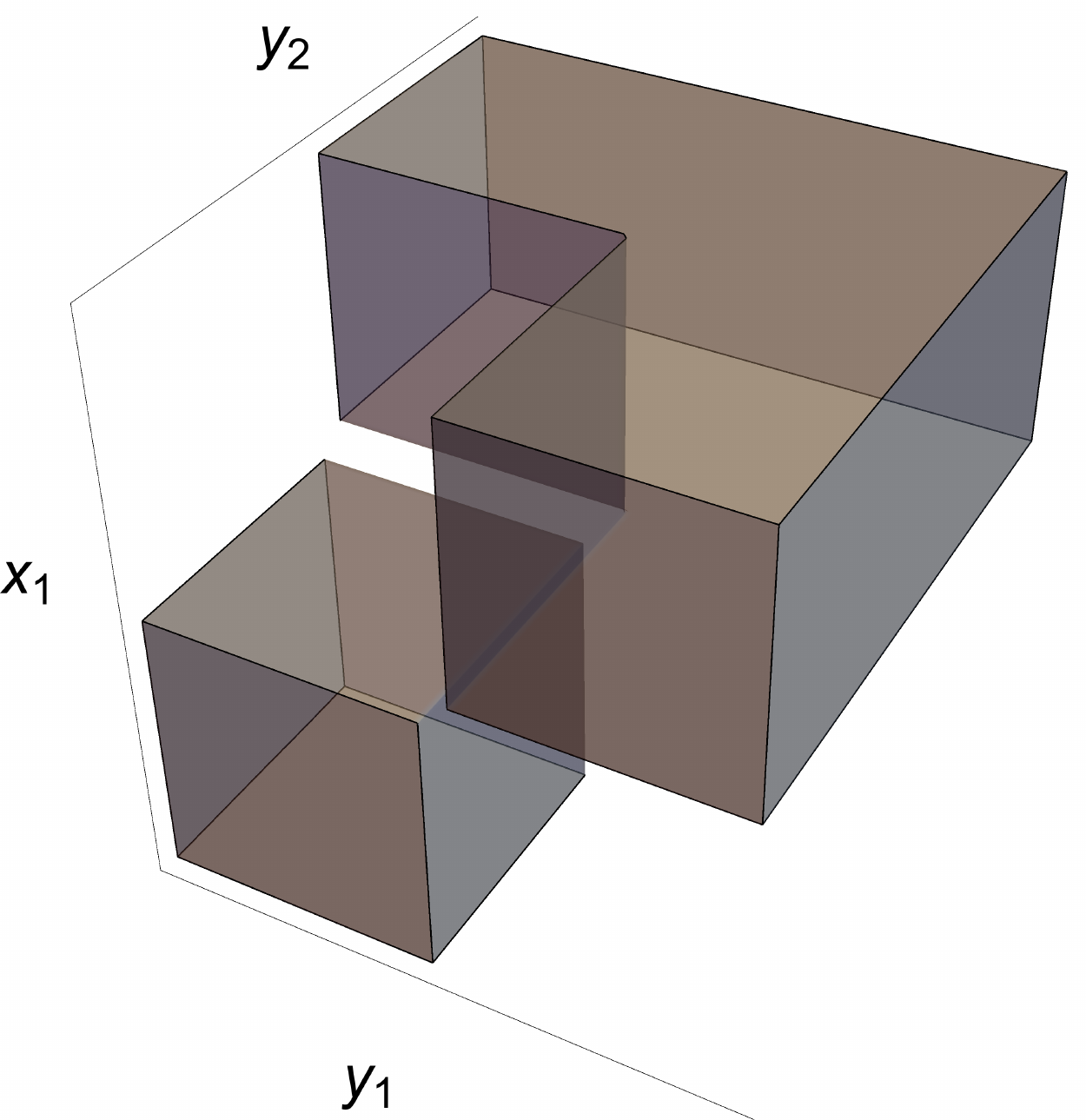}
    \includegraphics[width=.4\textwidth]{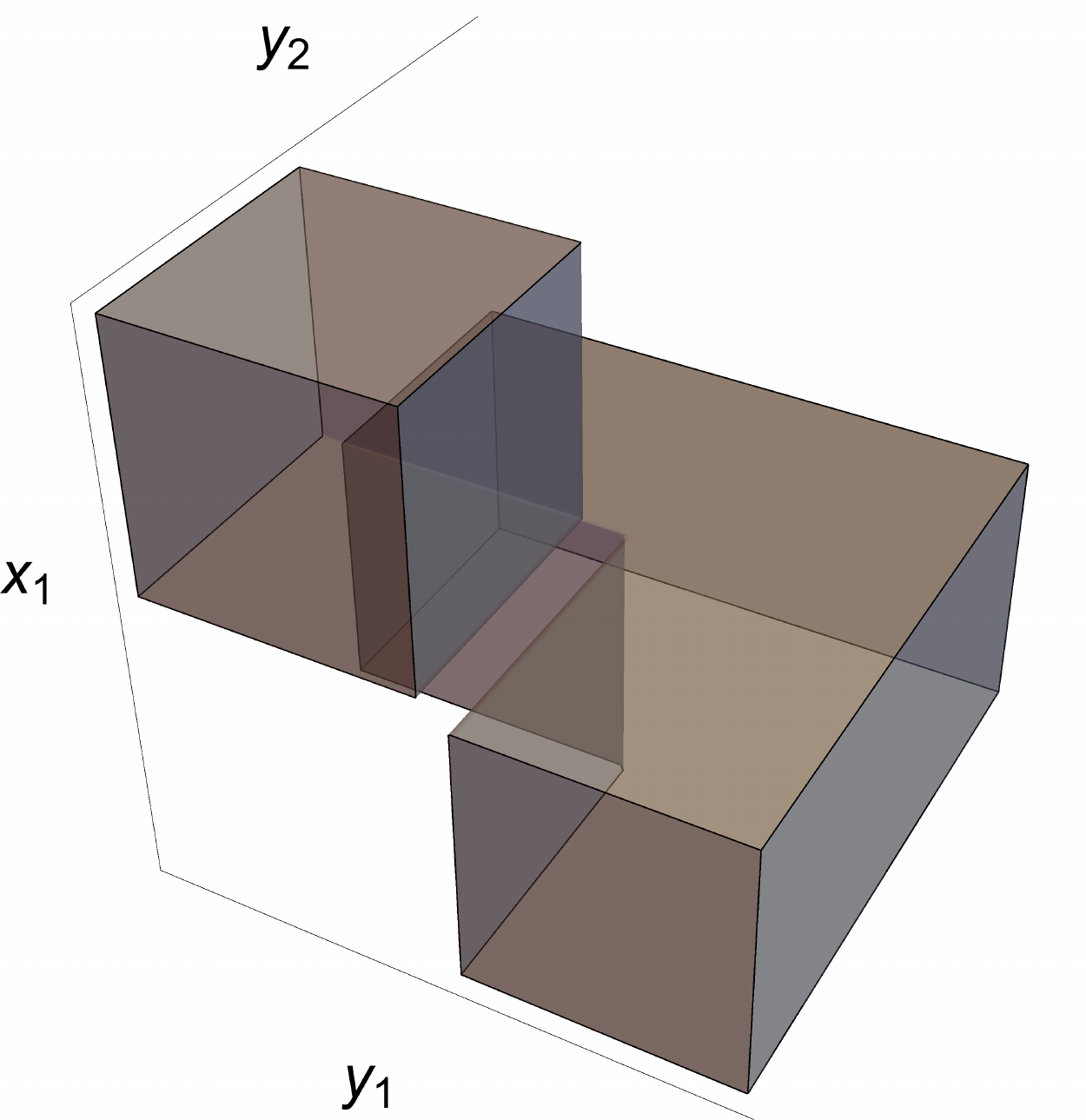}
    \caption{Concordant examples}\label{fig:tau-p-concord-1}
  \end{subfigure}
  \begin{subfigure}[t]{0.33\textwidth}
    \centering
    \includegraphics[width=.8\textwidth]{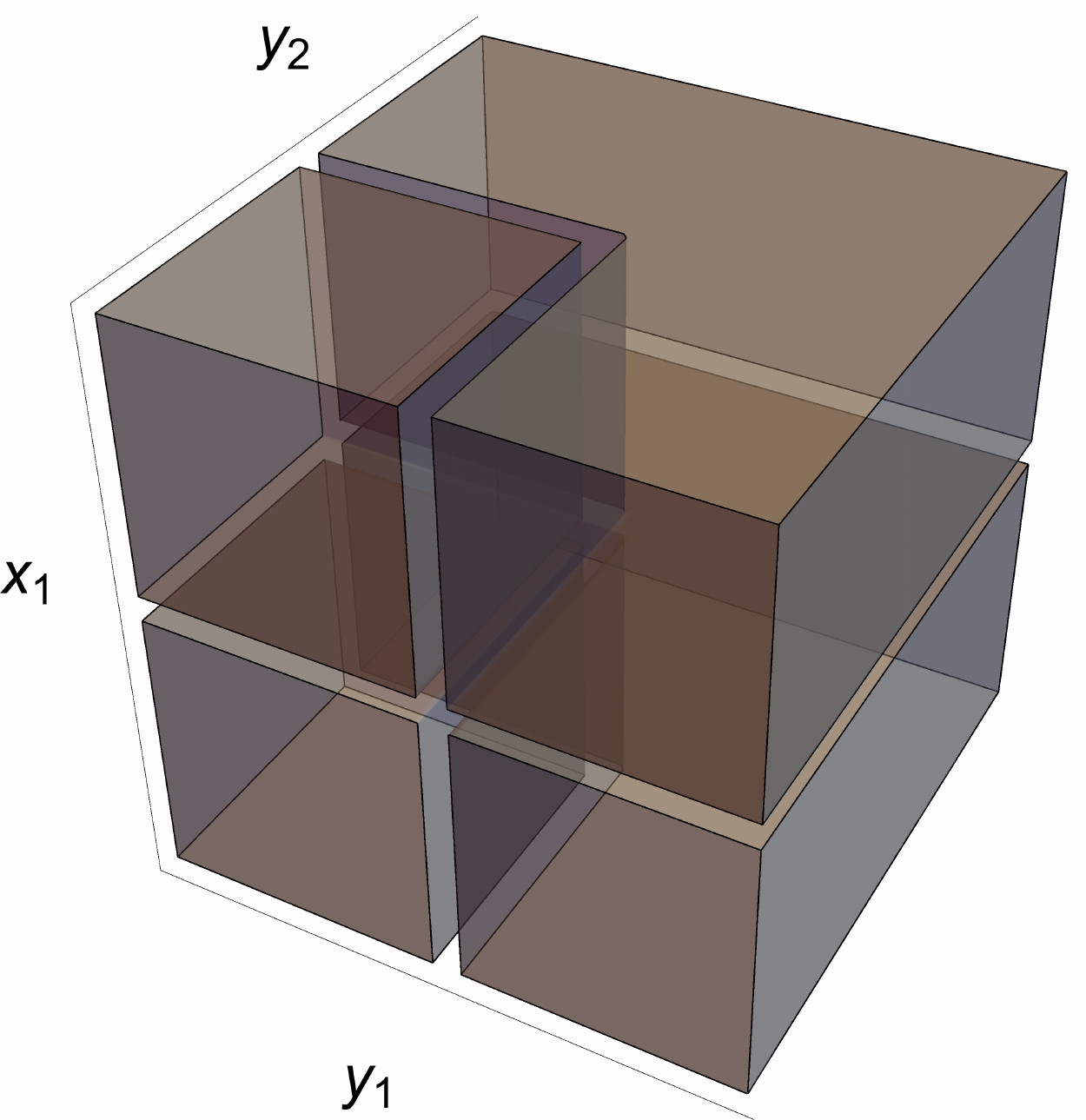}
    \caption{Discordant example}\label{fig:tau-p-discord}
  \end{subfigure}
  \caption{
  The bivariate sign covariance $\tau^*$ can be defined in terms of the probability of concordance and discordance of four points in $\mathbb{R}^2$ \citep[Figure 3]{BergsmaDassios14}. Our multivariate extension $\tau^*_P$ is based on higher-dimensional generalizations of concordance and discordance.  For illustration, let $x^1,...,x^4\in\bR$ and $y^1,...,y^4\in\bR^2$.  Considering either plot in panel (a), if precisely two tuples $(x^i,y^i)$ fall in each of the two gray regions, then the four tuples are concordant for $\tau^*_P$, but other types of concordance exist.  Considering panel (b), if exactly one $(x^i,y^i)$ lies in each of the gray regions, here the lower two regions are just translated copies of the top regions, then the four tuples are discordant; again, other types of discordance exist. Unlike in the bivariate case, points may be simultaneously concordant and discordant with respect to $\tau^*_P$.
  }\label{fig:tau-p-concord-discord}
\end{figure}


\section{Preliminaries}\label{sec:preliminaries}

\subsection{Manipulating Random and Fixed Vectors}

We begin by establishing conventions and notation used throughout
the paper.  Let
\begin{align*}
  (Z_1,\dots, Z_{r+s}) = Z = (X,Y) = ((X_1,\dots,X_r), (Y_1,\dots,Y_s))
\end{align*}
be a random vector taking values in $\bR^{r+s}$, and let
$(X^i,Y^i)=Z^i$ for $i\in\bZ_{>0}$ be a sequence of independent and
identically distributed\ copies of $Z$.  When $X$ and $Y$ are
independent we write $X\indep Y$, otherwise we write
$X\cancel\indep Y$. We let $F_{XY},F_X,$ and $F_Y$ denote the
\emph{cumulative distribution functions} for $(X,Y)$, $X$, and $Y$,
respectively.

We will
require succinct notation to describe (permuted) tuples of vectors. 
For any $n\geq 1$, define $[n] = \{1,\dots,n\}$. Let
$w^1,\dots,w^n\in \bR^d$. Then for any
$i_1,\dots,i_m,\ j_1,\dots,j_k \in [n]$, let
\begin{align*}
  w^{i_1,\dots,i_m} &= w^{(i_1,\dots,i_m)} = (w^{i_1},\dots,w^{i_m}) \quad \text{and} \quad \\
  (w^{i_1,\dots,i_m}, w^{j_1,\dots,j_k}) &= (w^{i_1},\dots,w^{i_m}, w^{j_1},\dots,w^{j_k}).
\end{align*}
If $[n]$ appears in the superscript of a vector it should be interpreted as an ordered vector, that is, we let $w^{[n]} = w^{(1,\dots,n)} = (w^1,\dots,w^n)$.

Let $S_n$ be the symmetric group.  For $\sigma\in S_n$ and $w^{[n]}\in \bR^{d\times n}$, let
\begin{align*}
  \sigma w^{[n]} = (w^{\sigma^{-1}(1)},\dots,w^{\sigma^{-1}(n)}).
\end{align*} 
This defines a (left) group action of $S_n$ on $\bR^{d\times n}$ that we will encounter often. As our convention is that $[n]$ is a tuple when in a
superscript, we have that $\sigma w^{[n]} = w^{\sigma[n]}$
for all $w^{[n]}\in\bR^{d\times n}$.  We 
stress that
$\sigma(1,\dots,n) =(\sigma^{-1}(1),\dots,\sigma^{-1}(n)) \not=
(\sigma(1),\dots,\sigma(n))$ in general.

\subsection{Hoeffding's $D$}

The bivariate setting from~(\ref{eq:hoeffding-D-R2}) immediately extends to a
multivariate version of Hoeffding's $D$ for the random vectors $X$ and
$Y$ by defining
\begin{align*}
  D(X,Y) = \int_{\bR^r\times \bR^s} (F_{XY}(x,y) - F_X(x)F_Y(y))^2 \dF_{XY}(x,y).
\end{align*}
Since $X\indep Y$ if and only if $F_{XY}(x,y) = F_X(x)F_Y(y)$, it is
clear that $X\indep Y$ implies $D(X,Y) = 0$. The converse need not always be true as the next example shows.

\begin{example} \label{exmp:hoeffding-failure}
  Let $Z=(X,Y)$ be a bivariate distribution with $P(X=1,Y=0) = P(X=0,Y=1) = 1/2$. Then clearly $X$ and $Y$ are not independent but we have that
  \begin{multline*}
    D(X,Y) = \frac{1}{2}(F_{XY}(1,0) - F_{X}(1)F_Y(0))^2 + \frac{1}{2}(F_{XY}(0,1) - F_{X}(0)F_Y(1))^2 \\
           = \frac{1}{2}(1/2 - 1\cdot 1/2)^2 + \frac{1}{2}(1/2 - 1/2 \cdot 1)^2 \;=\; 0.
  \end{multline*}
\end{example}

Thus, $D(X,Y)$ is I-consistent but not D-consistent in general. It is,
however, consistent for a large class of continuous distributions.  

\begin{theorem}[Multivariate version of Theorem 3.1 in
  \citeauthor{Hoeffding48},
  \citeyear{Hoeffding48}] \label{thm:hoeffding-consistent} Suppose $X$
  and $Y$ have a continuous joint density
  $f_{XY} = \frac{\partial}{\partial x_1\ \dots\ \partial
    x_r\ \partial y_1\ \dots\ \partial y_s}F_{XY}$ and continuous
  marginal densities $f_X$ and $f_Y$. Then
  $D(X,Y) = 0$ if and only if $X\indep Y$.
\end{theorem}

\begin{proof}
  The bivariate case is treated in Theorem 3.1 in \cite{Hoeffding48}.  The proof of the multivariate case is analogous.
\end{proof}

Example \ref{exmp:hoeffding-failure} highlights that the failure of $D(X,Y)$ to detect all dependence structures can be attributed to the measure of integration $\text{d}F_{XY}$. This suggests the following modification of $D$ which we call \emph{Hoeffding's $R$},
\begin{align*}
  R(X,Y)
  &=\int_\bR^{r+s} (F_{XY}(x,y) - F_X(x)F_Y(y))^2 \prod_{i=1}^r\dF_{X_i}(x_i) \prod_{j=1}^s\dF_{Y_j}(y_j).
\end{align*}
We suspect that it is well known that $R$ is consistent
but we could not find a compelling reference of this fact. For completeness we include a proof in the appendices.

\begin{theorem}\label{thm:prod-consistent} 
  Let $(X,Y)$ be drawn from a multivariate distribution on $\bR^r\times \bR^s$ as usual. Then $R(X,Y) \geq 0$ and $R(X,Y) = 0$ if and only if $X\indep Y$.
\end{theorem}

\subsection{Bergsma--Dassios Sign-Covariance $\tau^*$} \label{sec:tau-star-intro}

\cite{BergsmaDassios14} defined $\tau^*$ only for bivariate distributions so let $r=s=1$ for this section. While $\tau^*$ has a natural definition in terms of concordant and discordant quadruples of points, we will present an alternative definition that will be more useful for our purposes. First for any $w^{[4]}\in\bR^4$ let $I_{\tau^*}(w^{[4]}) = 1_{[w^1,w^2<w^3,w^4]}$ where $w^1,w^2<w^3,w^4$ if and only if $\max(w^1,w^2) < \min(w^3,w^4)$. Then, as is shown by \cite{BergsmaDassios14}, we have that
\begin{align}
  \tau^*(X,Y) &= E\Big[\Big(I_{\tau^*}(X^{[4]}) + I_{\tau^*}(X^{4,3,2,1}) - I_{\tau^*}(X^{1,3,2,4}) - I_{\tau^*}(X^{4,2,3,1})\Big) \label{eq:tau-star}\\
  &\hspace{15mm}\cdot \Big(I_{\tau^*}(Y^{[4]}) + I_{\tau^*}(Y^{4,3,2,1}) - I_{\tau^*}(Y^{1,3,2,4}) - I_{\tau^*}(Y^{4,2,3,1})\Big) \Big].\nonumber
\end{align}
While \cite{BergsmaDassios14} conjecture that $\tau^*$ is consistent for all bivariate distributions, the proof of this statement remains elusive. The current understanding of the consistency of $\tau^*$ is summarised by the following theorem.

\begin{theorem}[Theorem 1 of \citeauthor{BergsmaDassios14}, \citeyear{BergsmaDassios14}]\label{thm:tau-star-consistency}
  Suppose $(X,Y)$ are drawn from a bivariate continuous distribution, discrete distribution, or a mixture of a continuous and discrete distribution. Then $\tau^*(X,Y)\geq 0$ and $\tau^*(X,Y) = 0$ if and only if $X\indep Y$.
\end{theorem}

Theorem \ref{thm:tau-star-consistency} does not apply to any singular distributions; for instance, we are not guaranteed that $\tau^*>0$ when $(X,Y)$ are generated uniformly on the unit circle in $\bR^2$.

\section{Symmetric Rank Covariance} \label{sec:src}

\subsection{Definition and Examples}

We now introduce a new class of nonparametric dependence measures that
depend on $X$ and $Y$ only through their joint ranks.

\begin{definition}[Matrix of Joint Ranks] \label{def:joint-rank-matrix}
  Let $w^{[m]}\in\bR^{d\times m}$. Then the \emph{joint rank matrix} of $w^{[m]}$ is the $[m]$-valued $d\times m$ matrix with $i,j$ entry
  \begin{align*}
    \cR(w^{[m]})_{ij} = 1 + \sum_{k=1}^m 1_{[w^{k}_i < w^{j}_i]},
  \end{align*}
  that is, $\cR(w^{[m]})_{ij}$ is the rank of $w^j_i$ among $w^1_i,\dots,w^m_i$ for $i\in[d]$.
\end{definition}

\begin{definition}[Rank Indicator Function]
  A \emph{rank indicator function of order $m$ and dimension $d$} is a
  function $I:\bR^{d\times m}\to \{0,1\}$ such that
  $I(\cR(w^{[m]})) = I(w^{[m]})$ for all $w^{[m]}\in\bR^{d\times m}$.
  In other words, $I$ depends on its arguments only through their
  joint ranks.
\end{definition}

\begin{definition}[Symmetric Rank
  Covariance]\label{def:symmetric-rank-covariance}
  Let $I_X$ and $I_Y$ be rank indicator functions that have equal
  order $m$ and are of dimensions $r$ and $s$, respectively.  Let $H$
  be a subgroup of the symmetric group $S_m$ with an equal number of even 
  and odd permutations.  Define
  \begin{align}
    \mu_{I_X,I_Y,H}(X,Y) = E\Big[\Big(\sum_{\sigma \in H}
    \sign(\sigma)\ I_{X}(X^{\sigma[m]})\Big)\ 
    \Big(\sum_{\sigma \in H}\sign(\sigma)\ 
    I_{Y}(Y^{\sigma[m]})\Big)\Big].  \label{eq:src} 
  \end{align}
  Then a measure of dependence $\mu$ is a \emph{Symmetric Rank
  Covariance} if there is a scalar $c>0$ and  a triple $(I_X,I_Y,H)$ as specified above such that $\mu=c\ \mu_{I_X,I_Y,H}$.  
  More generally, $\mu$
  is a \emph{Summed Symmetric Rank Covariance} if it is
  the sum of several Symmetric Rank Covariances.
 \end{definition}
 
Some of the symmetric rank covariances we consider have the two rank indicator functions equal, so $I_X=I_Y=I$.  In this case, we also use the abbreviation $\mu_{I,H}=\mu_{I,I,H}$.

\begin{remark}
Recall from Section \ref{sec:tau-star-intro} that for any $z^{[4]}\in \bR^4$ we write $z^1,z^2 < z^3,z^4$ to mean $\max(z^1,z^2) < \min(z^3,z^4)$. To simplify the definitions of rank indicator functions we generalise this notation as follows. Let $\sim$ be any binary relation on $\bR^d$. Then for $z^{[l]}\in\bR^{d\times l}$ and $w^{[k]}\in\bR^{d\times k}$ we write $z^1,\dots,z^l \sim w^{1},\dots,w^k$ to mean $z^i \sim w^j$ for all $(i,j)\in[l]\times [k]$.
\end{remark}

It is easy to show that many existing nonparametric measure of dependence are Symmetric Rank Covariances.

\begin{proposition}\label{prop:standard-measures-are-srcs}
	Let $X$ and $Y$ take values in $\bR^r$ and $\bR^s$, respectively.  Consider the permutation groups $H_{\tau} = \langle (1\ 2) \rangle$ and $H_{\tau^*} = \langle (1\ 4), (2\ 3) \rangle$.   \smallskip
    
    (i) Bivariate case ($r=s=1$):   Kendall's $\tau$, its square $\tau^2$, and $\tau^*$ of Bergsma--Dassios are Symmetric Rank Covariances. Specifically, 
    \begin{align*}
     \tau=\mu_{I_{\tau},H_{\tau}}, \quad \tau^2= \mu_{I_{\tau^2},H_{\tau^*}}, \quad \text{and}\quad \tau^*=\mu_{I_{\tau^*},H_{\tau^*}},
  \end{align*}
   where the one-dimensional rank indicator functions are defined as
   \begin{align*}
      I_{\tau^*}(w^{[4]}) = I_{[w^1,w^2<w^3,w^4]}, \quad
     I_{\tau}(w^{[2]}) = I_{[w^1<w^2]}, \quad \text{and} \quad I_{\tau^2}(w^{[4]}) = I_{\tau}(w^{1,4})I_{\tau}(w^{2,3}).
  \end{align*}
       
  (ii) General case ($r,s\ge 1$): Both $D$ and $R$ are Symmetric Rank Covariances.  Specifically,
  \begin{align*}
     D=\frac{1}{4}\mu_{I_{D,r},I_{D,s},H_{\tau^*}} \quad \text{and} \quad  R=\frac{1}{4}\mu_{I_{R,r},I_{R,s},H_{\tau^*}}
  \end{align*}
   where for any dimension $d\ge 1$ and $w^1,w^2,\dots \in \bR^d$, we define
  \begin{align*}
     I_{D,d}(w^{[5]}) &= I_{[w^1, w^2\preceq w^5]}I_{[w^3,w^4\not\preceq w^5]},  &
      I_{R,d}(w^{[4+d]}) &= \prod_{i=1}^{d}I_{[w^1_i,w^2_i \leq w_i^{4+i} < w^3_i,w^4_i]}
  \end{align*}
  with $w^i\preceq w^j$ if and only if $w^i_\ell\leq w^j_\ell$ for all $\ell \in [d]$.
\end{proposition}

\begin{remark}
  The bivariate dependence measure Spearman's $\rho$ can
  be written as
  \begin{align*}
    \rho(X,Y) &= 6\ E[ I_{[X^1<X^2<X^3]} \ \big(I_{[Y^1<Y^2<Y^3]}+ I_{[Y^1<Y^3<Y^2]} +I_{[Y^2<Y^1<Y^3]} \\
 &\hspace{42mm} - I_{[Y^3<Y^1<Y^2]} - I_{[Y^2<Y^3<Y^1]} - I_{[Y^3<Y^2<Y^1]}\big)].
  \end{align*}
  In light of Lemma \ref{lem:alternate-src-forms} below, one 
  might expect
  $\rho$ to be a Symmetric Rank Covariance. However, upon examing
  which of the above indicators are negated, one quickly notes that
  the permutations do not respect the $\sign$ operation of the
  permutation group $S_3$.  For instance,  
  $I_{[Y^1<Y^2<Y^3]}$ and $I_{[Y^1<Y^3<Y^2]}$ are related through a single
  transposition and yet the terms have the same sign above. While it seems difficult
  to prove conclusively that $\rho$ is not a Symmetric Rank
  Covariance, this suggests that it is
  not. Somewhat surprisingly however, $\rho$ is a Summed Symmetric
  Rank Covariance which can be seen by expressing $\rho$ as
  \begin{align*}
    \rho(X,Y) = 3\ E\Big(b(X^{[3]})b(Y^{[3]}) + b(X^{[3]})b(Y^{[1,3,2]}) +b(X^{[3]})b(Y^{[2,1,3]})\Big)
  \end{align*}
  where $b(z^{[3]}) = I_{[z^1<z^2<z^3]} - I_{[z^3<z^2<z^1]}$ for all $z^i\in\bR$.
\end{remark}

\subsection{General Properties}

While many interesting properties of Symmetric Rank Covariances
depend on the choice of group $H$ and indicators $I_X,I_Y$, there are
several properties which hold for all such choices.

\begin{proposition} \label{prop:general-src-results} Let $\mu$ be
  Symmetric Rank Covariance. Then $\mu$ is nonparametric and
  I-consistent.  If $\nu$ is another Symmetric Rank Covariance, then
  so is the product $\mu\nu$.
\end{proposition}

The property for products in particular justifies squaring Symmetric Rank Covariances, as was
done for bivariate rank correlations in \cite{Leung2016}. Later, it will be useful to express a Symmetric Rank Covariances in an equivalent form.

\begin{lemma}\label{lem:alternate-src-forms}
  In reference to Equation \eqref{eq:src}, we have 
  \begin{align}
    \mu_{I_X,I_Y,H}(X,Y) &= |H|\ E\Big[I_X(X^{[m]})\ \Big(\sum_{\sigma\in H}\sign(\sigma)\ I_Y(Y^{\sigma[m]})\Big)\Big] \label{eq:mu-with-x}\\
    &= |H|\ E\Big[I_Y(Y^{[m]})\ \Big(\sum_{\sigma\in H}\sign(\sigma)\ I_X(X^{\sigma[m]})\Big)\Big].\label{eq:mu-with-y}
  \end{align}
\end{lemma}


\section{Generalizing Hoeffding's D and $\tau^*$}\label{sec:generalizing}

\subsection{Discretization Perspective} \label{sec:discretization}

In this section we introduce a collection of Summed Symmetric Rank
Covariances that are consistent and can be regarded as natural generalizations
of Hoeffding's $D$ and $R$. We begin by showing that
$D$ and $R$ are accumulations of, possibly infinitely many,
independence measures between binarized versions of $X$ and $Y$.  

\begin{definition}\label{def:binarization}
  Let $Z=(Z_1,\dots,Z_q)$ be an $\bR^d$-valued random vector, and let
  $z\in\bR^d$.  The \emph{binarization of $Z$ at $z$} is the random
  vector
  \begin{align*}
    B^Z(w) = (1_{[Z_1> z_1]},\dots,1_{[Z_d> z_d]}).
  \end{align*}
  We call $z$ the \emph{cutpoint} of the binarization.
\end{definition}

For any $z=(x,y)\in\bR^{r+s}$ we have $B^Z(z) = (B^X(x),
B^Y(y))$. Clearly $B^Z(z)$, $B^X(x)$, $B^Y(y)$ are discrete random
variables taking values in $\{0,1\}^{r+s},\ \{0,1\}^r$, and $\{0,1\}^s$ respectively. The
cutpoint $z$ divides $\bR^r\times \bR^s$ into $2^{r+s}$ orthants
corresponding to the states of $B^Z(z)$.   We index these orthants by
vectors $\ell \in \{0,1\}^{r+s}$, and define
\begin{align*}
  p(z)_{\ell} = P(B^Z(z) = \ell) = P(Z_i \lesseqgtr_{\ell_i} z_i,\ i\in [r+s]) \hspace{3mm} \text{where} \hspace{3mm} \lesseqgtr_{\ell_i}\ \equiv \twopartdef{\leq}{\ell_i=0,}{>}{\ell_i=1.}
\end{align*}
Let $p(z)$ be the $2\times \dots\times 2=2^{r+s}$ tensor with
coordinates $p(z)_\ell$.  Independence between $B^X(x)$ and $B^Y(y)$
can be characterised in terms of the rank of a flattening, or
matricization, of $p(z)$.  Let $M(x,y)$ be the real $2^r\times 2^s$
matrix with entries
\begin{align*}
  M(x,y)_{\ell_X\ell_Y} = p(z)_{\ell_X\ell_Y}
\end{align*}
for indices $\ell_X\in\{0,1\}^r$, $\ell_Y\in\{0,1\}^s$ that are
concatenated to form $\ell_X\ell_Y$.  It then holds that
$B^X(x)\indep B^Y(y)$ if and only if $M(x,y)$ has rank 1
\citep[Chapter 3]{oberwolfach}.

The matrix $M(x,y)$ has rank 1 if and only if all of its $2\times 2$
minors vanish, that is, for all $\ell_X,\ell'_X\in\{0,1\}^r$
and $\ell_Y,\ell'_{Y}\in\{0,1\}^s$ we have
\begin{align*}
  0&=M(x,y)_{\ell_X\ell_Y}M(x,y)_{\ell'_X\ell'_Y} -
     M(x,y)_{\ell'_X\ell_Y}M(x,y)_{\ell_X\ell'_Y}  \\
    &= p(z)_{\ell_X\ell_Y}p(z)_{\ell'_X\ell'_Y} -
      p(z)_{\ell'_X\ell_Y}p(z)_{\ell_X\ell'_Y}.
\end{align*}
One may easily show that $X\indep Y$ if and only if $B^X(x)\indep B^Y(y)$ for all $x,y$. This suggests defining a measure of dependence equal to the integral of the sum of squared minors of the above form. To recover both $D$ and $R$, however, we will need to generalise slightly by considering $2\times 2$ \emph{block minors} defined below. These block minors correspond to the fact that $B^X(x)\indep B^Y(y)$ if and only if $P(B^X(x)\in L,\ B^Y(y)\in R) = P(B^X(x)\in L)\ P(B^Y(y)\in R)$ for all $L\subset\{0,1\}^r,\ R\subset\{0,1\}^s$.

\begin{definition}\label{def:block-minor}
  Let $L,L' \subset\{0,1\}^r$ and $R,R' \subset\{0,1\}^s$ be nonempty subsets with $L\cap L'=\emptyset$ and $R\cap R'=\emptyset$. Then the \emph{$2\times 2$ block minor of $M(x,y)$ along $(L,L',R,R')$} is the value
  \begin{align*}
    &(\sum_{\substack{\ell_X\in L \\ \ell_Y \in R}} p(z)_{\ell_X\ell_Y})(\sum_{\substack{\ell_X\in L' \\ \ell_Y \in R'}} p(z)_{\ell_X'\ell_Y'}) - (\sum_{\substack{\ell_X'\in L' \\ \ell_Y \in R}} p(z)_{\ell_X'\ell_Y})(\sum_{\substack{\ell_X\in L \\ \ell_Y \in R'}} p(z)_{\ell_X\ell_Y'}) \\
    &= \sum_{\substack{\ell_X\in L \\ \ell_Y \in R}}\sum_{\substack{\ell_X\in L \\ \ell_Y \in R}}(p(z)_{\ell_X\ell_Y}p(z)_{\ell_X'\ell_Y'} - p(z)_{\ell_X'\ell_Y}p(z)_{\ell_X\ell_Y'}).
  \end{align*}
\end{definition}

\begin{proposition}\label{prop:block-minors}
$B^X(x)\indep B^Y(y)$ if and only if all $2\times 2$ block minors of $M(x,y)$ vanish.
\end{proposition}

If $L,L',R,R'$ are singletons the $2\times 2$ block minor reduces to a usual $2\times 2$ minor. 

We now propose to assess dependence by integrating squared block minors.  The integration measures we allow are derived from the variables' joint distribution but may be taken to be products of marginals as encountered for the measure of dependence $R$.

\begin{definition}[Integrated Squared Minor] \label{def:ism}
  For any $d\geq 0$ let $0_d\in \bR^d$ be the vector of all zeros. A measure $\mu(X,Y)$ is called an \emph{integrated squared minor} if there exists $L\subset \{0,1\}^r\setminus \{0_r\}$, $R\subset \{0,1\}^s\setminus \{0_s\}$, $E_1,\dots,E_t$ partitioning $[r]$, and $F_1,\dots,F_{t}$ partitioning $[s]$, such that
  \begin{align*}
    \mu(X,Y) = \int_{\bR^{r+s}} A(x,y)^2\dlambda_{XY}(x,y)
  \end{align*}
  where $A(x,y)$ is the $2\times 2$ block minor of $M(x,y)$ along $(\{0_r\}, L, \{0_s\},R)$, and the cumulative distribution function $\lambda_{XY}$ can be written as
  \begin{align*}
    \lambda_{XY}(x,y) = \prod_{1\leq i\leq t} F_{X_{E_i}Y_{F_i}}(x_{E_i},y_{F_i}).
  \end{align*}
\end{definition}

As the next proposition shows, all integrated square minor measures are Symmetric Rank Covariances.

\begin{proposition}\label{prop:isms-are-srcs}
  Let $\mu$ be an Integrated Squared Minor as in Definition \ref{def:ism}, then $\mu$ is a Symmetric Rank Covariance. In particular, we have $\mu=\frac{1}{4}\mu_{I_X,I_Y,H}$ where $H = H_{\tau^*}$ and
  \begin{align*}
    I_X(w^{[4+t]}) &= \sum_{\ell^X\in L}\prod_{i=1}^t1_{[w^1_{E_i},w^2_{E_i}\ \preceq\ w^{4+i}_{E_i}]}\prod_{j\in E_i} 1_{[w_j^3,w_j^4\ \lesseqgtr_{\ell^X_j}\ w_j^{4+i}]} \quad (w^{[4+t]}\in\bR^{r\times (4+t)}), \\
    I_Y(w^{[4+t]}) &= \sum_{\ell^Y\in R}\prod_{i=1}^t1_{[w^1_{F_i},w^2_{F_i}\ \preceq\ w^{4+i}_{F_i}]}\prod_{j\in F_i} 1_{[w_j^3,w_j^4\ \lesseqgtr_{\ell^Y_j}\ w_j^{4+i}]} \quad (w^{[4+t]}\in\bR^{s\times (4+t)}).
  \end{align*}
  Moreover, if $L= \{0,1\}^r\setminus \{0_r\}$ and $R =\{0,1\}^s\setminus \{0_s\}$ we have that $\mu = D$ when $\lambda_{XY}=F_{XY}$ and $\mu = R$ when $\lambda_{XY}=F_{X_1}\cdots F_{X_r}\cdot F_{Y_1}\cdots F_{Y_s}$. 
\end{proposition}

Finally we can identify a collection of D-consistent Summed Symmetric Rank Covariances.

\begin{proposition} \label{prop:consistent-ssrcs}
  Let $L_1,\dots,L_k\subset \{0,1\}^r\setminus\{0_r\}$ and $R_1,\dots,R_k\subset \{0,1\}^s\setminus\{0_s\}$ be two collections of nonempty sets. Suppose that the sets $L_i\times R_i$ are pairwise disjoint and form a partition of $(\{0,1\}^r\setminus\{0_r\}) \times (\{0,1\}^s\setminus\{0_s\})$. For all $i\in[k]$ let
  \begin{align*}
    \mu^{joint}_i(X,Y) = \int_{\bR^{r+s}}A_i(x,y)^2\dF_{XY}(x,y)
  \end{align*}
  and
  \begin{align*}
    \mu^{prod}_i(X,Y) = \int_{\bR^{r+s}} A_i(x,y)^2\prod_{i=1}^r\text{d}F_{X_i}(x_i) \prod_{j=1}^s\text{d}F_{Y_j}(y_j)
  \end{align*}
  where $A_i(x,y)$ is the $2\times 2$ block minor along $(\{0_r\}, L_i, \{0_s\}, R_i)$. Then the Summed Symmetric Rank Covariance $\mu^{joint} = \sum_{i=1}^k\mu^{joint}_i$ is D-consistent in, at least, all cases that $D$ is; similarly $\mu^{prod} = \sum_{i=1}^k\mu^{prod}$ is D-consistent in all cases.
\end{proposition}

\subsection{Multivariate $\tau^*$} \label{sec:multivariate-tau-star}

Recall from Proposition \ref{prop:standard-measures-are-srcs} that $\tau^* = \mu_{I_{\tau^*},H_{\tau^*}}$. Multivariate extensions of $\tau^*$ should simultaneously capture the essential characteristics of $\tau^*$ while permitting enough flexibility to define interesting measures of high-order dependence. As a first step to distilling these essential characteristics, it seems natural that any multivariate extension of $\tau^*$ uses the same permutation subgroup $H_{\tau^*}$. 

\begin{remark}
There are 30 distinct subgroups of $S_4$ exactly 20 of which have an equal number of even and odd permutations and thus could be used in the definition of a Symmetric Rank Covariance. Given these many possible choices it may seem surprising that $H_{\tau^*}$ appears in the definition of so many existing measures of dependence, namely $\tau^*$, $\tau^2$, $D$, and $R$. Some intuition for the ubiquity of $H_{\tau^*}$ can be gleaned from the proof of Proposition \ref{prop:standard-measures-are-srcs} where we show that $H_{\tau^*}$ arises naturally from an expansion of $(F_{XY}(x,y)-F_X(x)F_Y(y))^2$.
\end{remark}

It now remains to find an appropriate generalization of $I_{\tau^*}$. To better characterise $I_{\tau^*}$ we require the following definition.

\begin{definition}[Invariance Group of an Indicator]
  Let $I$ be a rank indicator function of order $m$ and dimension $d$.   The permutations $\sigma\in S_{m}$ such that $I(\sigma w^{[m]}) = I(w^{[m]})$ for all $w^{[m]}\in\bR^{d\times m}$ form a group that we refer to as the \emph{invariance group $G$ of $I$}. For any Symmetric Rank Covariance $\mu_{I_X,I_Y,H}$, let $G_X,G_Y$ be the invariance groups of $I_X$ and $I_Y$ respectively. We then call \emph{$G = G_X\cap G_Y$ the invariance group of $\mu_{I_X,I_Y,H}$}.
\end{definition}

We now single out two properties of $I_{\tau^*}$.
\begin{property}
  $I_{\tau^*}$ is a rank indicator function of order $4$.
\end{property}
\begin{property}
  The invariance group of $I_{\tau^*}$ is $\langle (1\ 2), (3\ 4) \rangle$.
\end{property}
This inspires the following definition.

\begin{definition}\label{def:multivariate-tau-stars}
  We say that an Symmetric Rank Covariance $\mu_{I_X,I_Y,H}$ is a \emph{$\tau^*$ extension} if $I_X$ and $I_Y$ are rank indicators of order $4$ with invariance group $\langle (1\ 2),(3\ 4) \rangle$ and $H=H_{\tau^*}$.
\end{definition}

From the possible $\tau^*$ extensions we consider two notable candidates.

\begin{definition}\label{def:partial-tau-star}
  For any $d\geq 1$ let $I_{P}:\bR^{d\times 4}\to \{0,1\}$ be the rank indicator where for any $w^{[4]}\in \bR^{d\times 4}$ we have $I_P(w^{[4]}) = I_{[w^3,w^4\not\preceq w^1, w^2]}$. We then call $\mu_{I_P,I_P,H_{\tau^*}}$ the \emph{multivariate partial $\tau^*$} and write $\tau^*_P = \mu_{I_P,I_P,H_{\tau^*}}$.
\end{definition}

The definition of $I_P$ is inspired by $I_D$, see Proposition \ref{prop:standard-measures-are-srcs}.

\begin{definition}\label{def:joint-tau-star}
  For any $d\geq 1$ let $I_J:\bR^{d\times 4}\to \{0,1\}$ be the rank indicator where for any $w^{[4]}\in \bR^{d\times 4}$ we have $I_J(w^{[4]}) = I_{[ w^1, w^2 \prec w^3,w^4]}$. We then call $\mu_{I_J,I_J,H_{\tau^*}}$ the \emph{multivariate joint $\tau^*$} and write $\tau^*_J = \mu_{I_J,I_J,H_{\tau^*}}$.
\end{definition}

Our definition of $\tau^*_J$ comes immediately from $\tau^*$ when replacing the total order $<$ with $\prec$, although this might be the most intuitive multivariate extension of $\tau^*$ it is easily seen to not be D-consistent as the next example shows. In both of the above definitions, the extensions reduce to being $\tau^*$ when $r=s=1$. 

\begin{example} \label{exmp:tau-j-inconsistent}
  Let $X = (X_1,\dots,X_r)$ where $r$ is even and $X_1,\dots,X_r \sim$ Bernoulli$(1/2)$ are independent. Now let $Y = \text{XOR}(X_1,\dots.,X_r)$, that is, let $Y = 1$ if $\sum_{i=1}^r X^r$ is odd and $Y= 0$ otherwise. Now letting $(X^i,Y^i)$ be independent and identically distributed replicates of $(X,Y)$, $I_{J}(X^{[4]}) = 1$ if and only if $X^1=X^2 = (0,\dots,0)$ and $X^3=X^4=(1,\dots,1)$. Thus $I_j(X^{[4]}) = 1$ implies that $Y^1 = Y^2 =Y^3=Y^4 = 0$, and hence that $\sum_{\sigma \in H_{\tau^*}}\sign(\sigma)\ I_J(Y^{\sigma[4]}) = 0$. Thus we have that $\tau^*_J(X,Y) = 0$ while $X\cancel\indep Y$.

  This behavior only occurs when $r$ is even.  If $r$ is odd, then $\tau^*_J(X,Y) = 2^{-4r + 2}$.
\end{example}

Unlike for $\tau^*_J$, we have yet to discover an example where $\tau^*_P(X,Y)$ is 0 when $X\cancel\indep Y$. This leads us to conjecture that $\tau_{P}^*$, like the subclass of measures from Section \ref{sec:discretization}, is D-consistent.


\section{Estimation via U-statistics} \label{sec:estimation-via-u-statistics}

\subsection{Standard Form of U-Statistics Estimating Symmetric Rank Covariances}

The definition of Symmetric Rank Covariances allows them to be easily estimated with U-statistics. We begin with a discussion on how the computational efficiency these U-statistics can often be improved by leveraging efficient data structures for performing orthogonal range queries. We then consider the asymptotic properties of our estimators. Somewhat surprisingly, the invariance group of an Symmetric Rank Covariance is tied to its estimator's asymptotic behavior under the null hypothesis of independence. Indeed, we will exhibit a collection of Symmetric Rank Covariances whose U-statistics are degenerate and thus have non-Gaussian asymptotic distributions. In the bivariate case, these new definitions will help show explicitly why the asymptotic distributions of the U-statistics corresponding to $D$, $R$, and $\tau^*$ have the same, up to scaling, asymptotic distribution when $X$ and $Y$ are continuous and independent, a behavior first observed by \cite{NandyEtAl16}.

In the below we will assume we have observed a sample of $n\geq 1$ independent and identically distributed replicates of $Z=(X,Y)$, we call these $Z^1,\dots,Z^n$. Let $\mu$ be a Symmetric Rank Covariance as in Equation \eqref{eq:src} so that
\begin{align*}
      \mu(X,Y) = E\Big[\Big(\sum_{\sigma \in H} \sign(\sigma)\ I_{X}(X^{\sigma[m]})\Big)\ \Big(\sum_{\sigma \in H}\sign(\sigma)\ I_{Y}(Y^{\sigma[m]})\Big)\Big].
\end{align*}

Let $\kappa:\bR^{(r+s)\times m}\to \bR$ be the \emph{symmetrised kernel function} defined by
\begin{align} \label{eq:sym-kernel}
  \kappa(z^1,\dots,z^{m}) = \frac{1}{m!}\sum_{\sigma\in S_{m}} k(z^{\sigma[m]})
\end{align}
where the \emph{unsymmetrised kernel function} $k:\bR^{(r+s)\times m}\to \bR$ is defined by
\begin{align*}
  k(z^{[m]}) = \Big(\sum_{\sigma \in H} \sign(\sigma)\ I_{X}(x^{\sigma[m]})\Big)\ \Big(\sum_{\sigma \in H}\sign(\sigma)\ I_{Y}(y^{\sigma[m]})\Big).
\end{align*}
Then we define, for $n\geq m$ and $z^{[n]}\in\bR^{d\times n}$,
\begin{align}\label{eq:u-statistic}
  U_\mu(z^{[n]}) = \frac{1}{{n \choose m}}\sum_{1 \leq i_1<\dots<i_{m} \leq n} \kappa(z^{i_1,\dots,i_{m}}).
\end{align}
We call $U_\mu$ the \emph{U-statistic corresponding to $\mu$}.  Clearly, $U_\mu(Z^{[n]})$ is unbiased for $\mu(X,Y)$. For computational friendliness we will sometimes rewrite $\kappa$ using the following proposition.

\begin{proposition} \label{prop:rewrite-sym-kernel}
 For any $z^{[m]}\in\bR^{d\times m}$ we have that
  \begin{align}
    \kappa(z^{[m]}) &= \frac{|H|}{m!}\sum_{\gamma\in S_m} I_X(x^{\gamma[m]}) \ \Big(\sum_{\sigma\in H}\sign(\sigma)\ I_Y(y^{\sigma\gamma[m]})\Big) \label{eq:rewrite-sym-kernel-in-x}\\
    &=\frac{|H|}{m!}\sum_{\gamma\in S_m} I_Y(y^{\gamma[m]}) \ \Big(\sum_{\sigma\in H}\sign(\sigma)\ I_X(X^{\sigma\gamma[m]})\Big). \label{eq:rewrite-sym-kernel-in-y}
  \end{align}
\end{proposition}

\subsection{Efficient Computation}

The U-statistics defined by Equation \eqref{eq:u-statistic} are a sum over ${n\choose m}$ elements and thus, assuming the kernel function $\kappa$ can be evaluated in $m$ time, na\"{\i}vely require $O(m\ n^{m})$ time to compute. While this may be feasible for small $m$ and $n$, it will quickly become computationally prohibitive for even moderate sample sizes. 
While subsampling can be used in such cases to approximate our statistics of interest, it is not always clear how many samples of which size should be taken to obtain acceptable approximation error and, when many such samples are needed, subsampling approximations need not be fast. 
Fortunately, when specializing to the U-statistics estimating $D,R$, $\tau^*_P$, and $\tau^*_J$, we show that the use of efficient data structures for computing orthogonal range queries can reduce the asymptotic run time of the computations. While our observations do not generalise to all Symmetric Rank Covariances, as rank indicators can be made complex, there appear to be many, for instance $\tau^2$, for which a similar approach can be used to reduce run time. For especially large samples, these efficient computational strategies could be combined with subsampling procedures to more rapidly achieve small levels of approximation error.

For the remainder of this section we assume to have observed data $z^{[n]}\in\bR^{d\times n}$. Moreover, to simplify our run-time analyses, we will assume that $d$ is bounded so that for any functions $f,g:\bN\to\bN$ and $h:\bN^2\to\bN$ we have that $O(f(d) + g(d)h(n,d)) = O(h(n,d))$.

As the above defined U-statistics depend on $z^{[n]}$ only through their joint ranks we will make the further assumption that $z^{[n]} = \cR(z^{[n]})\in[n]^{d\times n}$ so that we have transformed $z^{[n]}$ into its corresponding matrix of joint ranks. The computational effort of this procedure is $O(n\ \log_2(n))$ and, as none of the algorithms we will present run in time less than this, performing this preprocessing step does not change the overall analysis.

In the bivariate case, it follows easily from the discussion by \citet{Hoeffding48} that $D$ can be computed in $O(n\ \log_2(n))$ time while, more recently, it has been shown that $\tau^*$ can be computed in $O(n^2)$ time \citep{Heller2016,WeihsEtAl16}. These computational savings largely rely on the ability to efficiently perform \emph{orthogonal range queries}.

\begin{definition}[Orthogonal Range Query]
  Let $z^{[n]}\in\bR^{d\times n}$. Then we say that the question, ``how many $z^i$ lie in $B\subset \bR^d$?'' is an \emph{orthogonal range query on $\{z^1,\dots,z^n\}$} if $B = I_1\times\dots\times I_d$ and, for $1\leq i\leq d$, we have that $I_i$ is a 1-dimensional interval of the form $(l_i,u_i),\ [l_i,u_i),\ (l_i,u_i],$ or $[l_i,u_i]$ for some $l_i,u_i\in\bR$.
\end{definition}

As the next proposition shows, see \cite{Heller2016} for the bivariate case, using a simple dynamic programming approach one may easily construct an $n^{d}$ tensor so that any orthogonal range query on $z^{[n]}$ can be answered in $O(1)$ time.

\begin{proposition} \label{prop:tensor-orth-search}
  Let $z^{[n]}\in \bR^{d\times n}$ be such that $z^{[n]} = \cR(z^{[n]})$. Then let $A\in \bN^{(n+1)\times \dots \times (n+1)}$ be a $d$-dimensional tensor indexed by elements of $\{0,1,\dots,n\}^d$ with $(i_1,\dots,i_d)\in \{0,\dots,n\}^d$ entry equaling
  \begin{align*}
    A(i_1,\dots,i_d) = \sum_{i=1}^n 1_{[z^{i} = (i_1,\dots,i_d)]}
  \end{align*}
  so that $A(i_1,\dots,i_d)$ equals the number of elements $z^i$ with value $(i_1,\dots,i_d)$. Now define $B\in \bN^{(n+1)\times \dots \times (n+1)}$ recursively so that it has $(i_1,\dots,i_d)\in \{0,\dots,n\}^d$ entry
  \begin{align*}
    B(i_1,\dots,i_d) &= 0 \quad \text{ if any $i_j=0$ and,}\\
    B(i_1,\dots,i_d) &= A(i_1,\dots,i_d) + \sum_{s=1}^n\sum_{\substack{\ell\in \{0,1\}^d\setminus \{0_d\} \\ \sum_k\ell_k=s}} (-1)^{s+1}B((i_1,\dots,i_d)-\ell) \quad \text{if otherwise}.
  \end{align*}
  Then for any $l = (l_1,\dots,l_d), u = (u_1,\dots,u_d)\in \{0,...,n\}^d$ the orthogonal range query ``how many $z^i$ lie in $B=(l_1,u_1]\times \dots \times(l_d,u_d]$?'' equals
  \begin{align*}
    \sum_{\ell\in\{0,1\}^d} (-1)^{\sum_{j=1}^d \ell_j} B(l_1^{\ell_1}\ u_1^{1-\ell_1},\ \dots,\ l_d^{\ell_d}\ u_d^{1-\ell_d}).
  \end{align*}
  When $d$ is bounded and $B$ is given, the above sum takes $O(1)$ time to compute.
\end{proposition}

\begin{proof}
This follows by a straightforward application of the inclusion-exclusion principle.
\end{proof}

Unfortunately, the above tensor takes $O(n^d)$ time to construct and so, when $d\geq m$ this procedure already takes as long or longer than simply computing the U-statistic na\"{i}vely. In such cases, we find that the range-tree data structure provides a better balance between quickly computing orthogonal range queries and the effort required for its construction.

\begin{proposition}[Range-Tree Data Structure, \citeauthor{Berg2008}, \citeyear{Berg2008}]
  Let $z^{[n]}$$\in \bR^{d\times n}$. There exists a data structure, called a range-tree, which takes $O(n\ \log_2(n)^{d-1})$ time to construct and can answer any orthogonal range query on $z^{[n]}$ in $O(\log_2(n)^{d-1})$ time.
\end{proposition}

See Section 5 of \citet{Berg2008} for a detailed exposition on Range-Trees, along with a discussion of the above proposition, and orthogonal range queries in general. As, to the best of our knowledge, there exists no freely available completely general implementation of range-trees we make such an implementation freely available at \url{https://github.com/Lucaweihs/range-tree}.  Range-trees are closely related to binary search trees, such as Red-Black Trees, which have been previously used to efficiently compute the U-statistics corresponding to $\tau$ and $\tau^*$ \citep{Christensen2005,WeihsEtAl16}. Using these efficient data structures we obtain substantial run time savings.

\begin{table}[t]
  \def~{\hphantom{0}}
  \caption{The asymptotic run times of computing the U-statistics $U_D,U_R,U_{\tau^*_P},U_{\tau^*_J}$ on a sample  $z^{[n]}\in\bR^{d\times n}$ na\"{i}vely versus the more efficient methods described in Appendix \ref{app:efficient-computation}.}\label{table:computational-results}
  \begin{tabular}{ccc}
    \toprule
    & \multicolumn{2}{c}{Algorithm run time} \\ 
    U-Statistic & Using Efficient Orthogonal Range Queries & Na\"{i}ve \\
    \midrule
    $U_D$ & $O(n\log_2(n)^{d-1})$ & $O(n^5)$ \\
    $U_{R}$ & $O(n^d)$ & $O(n^{4+d})$ \\
    $U_{\tau^*_P}$ and $U_{\tau^*_J}$ & $O(n^2\log_2(n)^{2d-1})$ & $O(n^4)$ \\
    \bottomrule
  \end{tabular}
\end{table}

\begin{proposition} \label{prop:computational-results}
Table \ref{table:computational-results} lists the asymptotic run time of computing $U_D,U_R,U_{\tau^*_P},$ and $U_{\tau^*_P}$ when using the algorithms described in Appendix \ref{app:efficient-computation}.
\end{proposition}

\subsection{Null Asymptotics}\label{sec:null-asymptotics}

Determining the asymptotic distribution of $U_\mu$ under the null hypothesis of independence, that $X\indep Y$, requires an understanding of the functions
\begin{align*}
  \kappa_i(z^1,\dots,z^i) = E[\kappa(z^1,\dots,z^i,Z^{i+1},\dots,Z^{m})].
\end{align*}
To this end, we introduce some simplifying lemmas and propositions.

\begin{lemma} \label{lem:kernel-flip}
  Suppose that $X\indep Y$. Let $S\subset[m]$ and let $G$ be the invariance group corresponding to $\mu$. Partition $H$ into equivalence classes $E_1,\dots,E_t$ where $h,h'\in H$ are equivalent if there exists $g\in G$ such that $gh(i) = h'(i)$ for all $i\in S$. If each $E_i$ contains an equal number of even and odd permutations then for any $z_1,\dots,z_m\in\bR^{r+s}$,
  \begin{align*}
    E[k(W^{[m]})] &= 0
  \end{align*}
  where $W^i = z^i$ if $i\in S$ and $W^i = Z^i$ otherwise.
\end{lemma}

Lemma \ref{lem:kernel-flip} allows to identify conditions guaranteeing that $U_\mu$ is degenerate, that is, cases in which $\sqrt{n}(U_\mu - EU_\mu)$ converges to 0 in probability.

\begin{proposition} \label{prop:multivariate-degenerate}
 Suppose that the conditions of Lemma \ref{lem:kernel-flip} hold for $\mu$ whenever $S$ is a singleton set. If $X\indep Y$ then $\kappa_1 \equiv 0$ and thus $U_\mu$ is a degenerate U-statistic.
\end{proposition}

As an application of the above Lemma and Proposition we show, the known result, that $\tau^*,D,$ and $R$ are degenerate U-statistics under independence. Moreover, we show that their $\kappa_2$ functions take a simple form.

\begin{lemma} \label{lem:simple-sym-kernel}
  Let $I_X,I_Y$ be two rank indicators of order $m\geq 4$ and dimensions $r$ and $s$ respectively. Let $\mu=\mu_{I_X,I_Y,H_{\tau^*}}$ be a Symmetric Rank Covariance. Suppose that $X\indep Y$. If the invariance group of $\mu$ contains the subgroup $G=\langle (1\ 2), (3\ 4) \rangle$ then $\kappa_1(z^1) \equiv 0$ so that $U_{\mu}$ is a degenerate U-statistic and
  \begin{align*}
    \kappa_2(z^1,z^2) = \frac{4}{{m \choose 2}} E\Big[a_{I_X}(x^1,x^2,X^{3,\dots,m})\Big]\ E\Big[a_{I_Y}(y^1,y^2,Y^{3,\dots,m})\Big]
  \end{align*}
  where for any rank indicator $I$ of order $m\geq 4$ we define
  \begin{align*}
    a_{I}(w^{[m]}) = \sum_{\sigma\in H_{\tau^*}}\sign(\sigma)I(w^{\sigma[m]}).
  \end{align*}
  As we have shown previously, $\tau^*,D,$ and $R$ satisfy the above conditions. Moreover, by construction, so do all multivariate $\tau^*$ extensions.
\end{lemma}

As was noted by \cite{NandyEtAl16}, the U-statistics corresponding to
$\tau^*$ and $D$ have, up to a scale multiple, the same asymptotic
distribution under the null hypothesis that $X\indep Y$ and $(X,Y)$
are drawn from a continuous bivariate distribution. We give a simple
proof of this fact, as well as showing that $U_R$ has asymptotic
distribution also a scale multiple of the others, and clarify the
constant multiple by which they differ.

\begin{proposition} \label{prop:asymptotic-dists-for-bivariate}
  Let
  \begin{align*}
    Z = \sum_{i=1}^\infty\sum_{j=1}^\infty \frac{1}{i^2j^2}(\chi_{1,ij}^2-1)
  \end{align*}
  where $\{\chi^2_{1,ij} : i,j \in \bN_+ \}$ is a collection of independent and identically distributed $\chi^2_1$ random variables. Then
  \begin{align*}
    n\ U_{\tau^*} &\to \frac{36}{\pi^4} Z \qquad\text{ and both} \qquad
    n\ U_{D},\ n\ U_{R} \to \frac{1}{\pi^4} Z
  \end{align*}
  in distribution.
\end{proposition}

To better understand at which sample size $n$ the finite sample
distributions of $U_{\tau^*}, U_{D}$, and $U_R$ become well
approximated by their asymptotic distributions we plot the total
variation distance between kernel density estimates of the finite
sample distributions of $U_{\tau^*}, U_{D}$, and $U_R$ for
$n\in\{15,30,60,120,240\}$ against the probability density functions
of their asymptotic distributions in Figure
\ref{fig:2-dim-asymp-dists}.  We observe good agreement even for
$n=60$. Unfortunately clarifying the exact asymptotic behavior in
higher dimensions appears to be significantly more difficult than in
the bivariate case. In part this is due to the fact that, unlike 
in the continuous bivariate case, the distributions of the random
vectors $X$ and $Y$ influence the asymptotic properties of our multivariate
U-statistics.
Indeed, even when $r=1$ and $s=2$ and $X$, $Y$ are normally distributed,
Figure \ref{fig:tau-star-j-asymp-dists} suggests that the correlation 
between $Y_1$ and $Y_2$ has a substantial impact on large sample
behavior. Because of these difficulties, we leave this problem for future work.

\begin{figure}[t]
  \centering
  \includegraphics[width=0.5\textwidth]{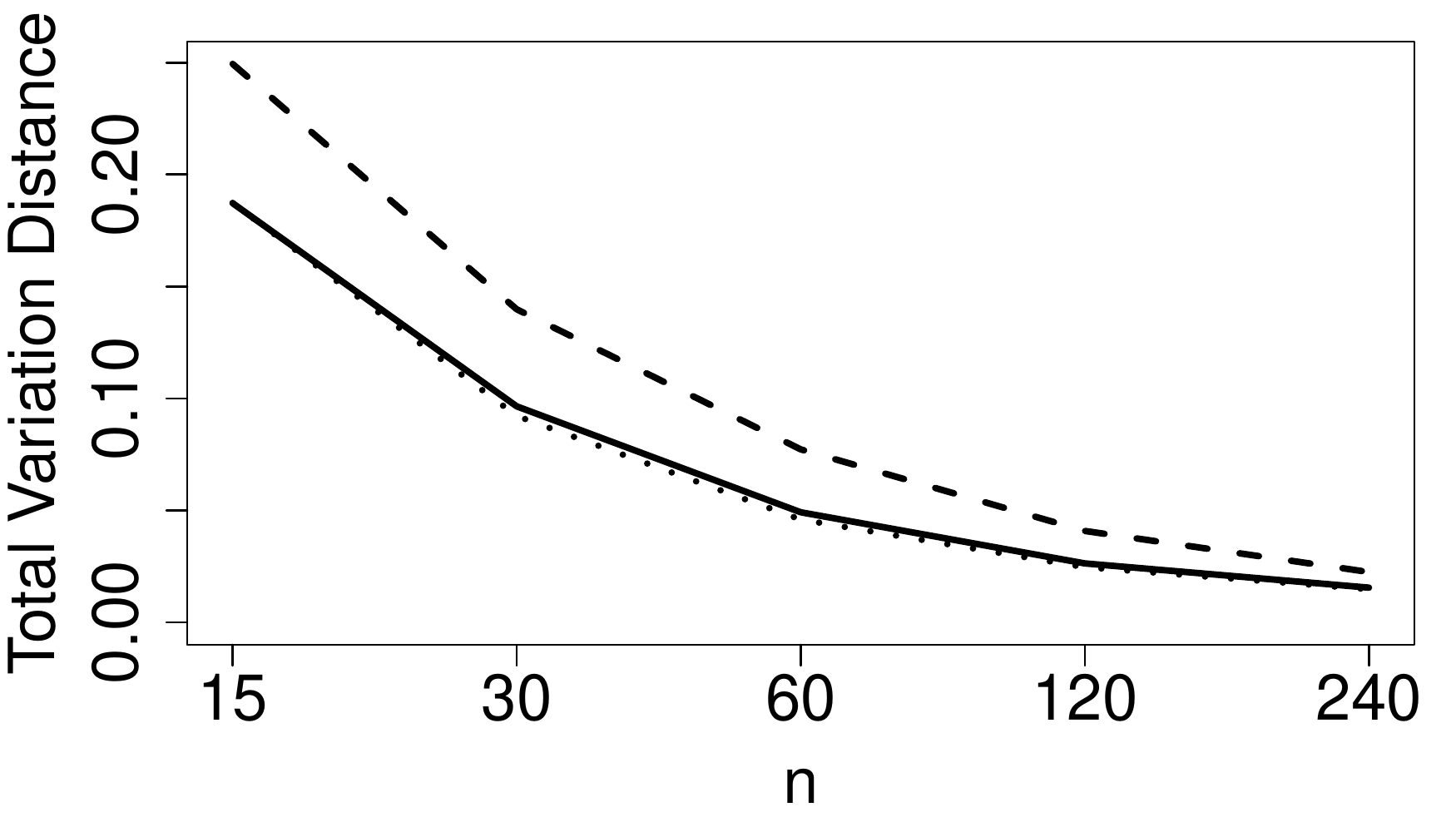}
  
  \caption{The total variation distance from kernel density estimators of the finite sample distributions of $U_{\tau^*}$ (solid line), $U_{D}$ (dashed line), and $U_R$ (dotted line) to the probability density functions of their asymptotic distributions. The x-axis is plotted on a log-scale. Here $n\in\{15,30,60,120,240\}$ is the sample size. The finite sample distributions are quite close to the asymptotic distributions even when $n$ is only $\approx 60$.}
  \label{fig:2-dim-asymp-dists}
\end{figure}

\begin{figure}[t]
  \centering
  \includegraphics[width=0.5\textwidth]{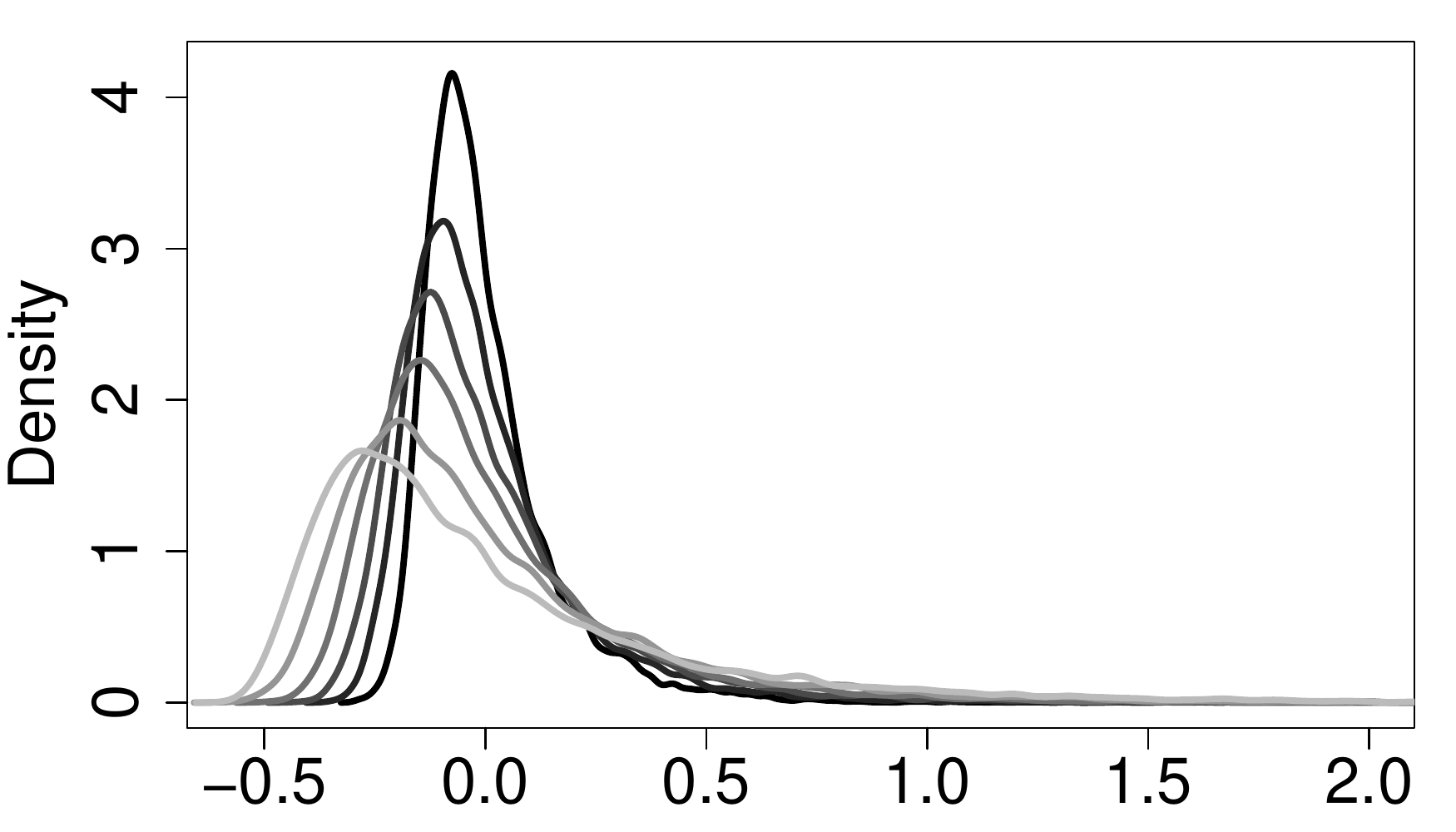}
  
  \caption{Kernel density estimates of the finite sample distributions of $n\ U_{\tau^*_J}$ for samples of size $n=70$ taken from $(X,Y)$ where $X,Y_1,Y_2\sim N(0,1)$, $(Y_1,Y_2)$ are jointly normal with correlation $\rho$, and $X\indep Y$. Here $\rho$ varies in $\{0, 1/5, \dots, 1\}$ with the lighter colored lines corresponding to kernel density estimates for smaller $\rho$. The large impact of $\rho$ on the finite sample distributions suggests that these differences carry over into the respective asymptotic distributions.}
  \label{fig:tau-star-j-asymp-dists}
\end{figure}


\section{Simulations}\label{sec:simulations}

\subsection{Power} \label{sec:power}

All of the following experiments are run in R \citep{R} using the package \texttt{SymRC} which can be obtained from \url{https://github.com/Lucaweihs/SymRC}. \texttt{SymRC} was created with efficiency in mind, all of the ``heavy lifting'' is done using C++ \citep{C++}. 

We consider the problem of testing if some univariate response $Y$ is independent of a set of covariates $X=(X_1,\dots,X_r)$. Our tests will be based on the U-statistics corresponding to $D,R, \tau^*_P$, and $\tau^*_J$. As explicit asymptotic distributions for $U_D,U_R, U_{\tau^*_P},$ and $U_{\tau^*_J}$ are not known we will use permutation tests. Unfortunately the computational complexity of $U_R, U_{\tau^*_P},$ and $U_{\tau^*_J}$ are such that, while it is certainly possible to perform permutation tests for a single moderately sized sample, it becomes computationally prohibitive to perform the many thousand such tests needed for Monte Carlo approximation of power.  We thus approximate the results of permutation tests: first we create a reference distribution for our U-statistic of interest under $X\indep Y$ by, for $R=1000$ iterations, randomly generating $x^1,\dots,x^n$ independently from the marginal distribution of $X$ and $y^1,\dots.,y^n$ independently from the marginal distribution of $Y$ and saving the value of the U-statistic for this data set. For an independent and identically distributed sample $\cD = \{(\ol{x}^1,\ol{y}^1),\dots,(\ol{x}^n,\ol{y}^n)\}$ from the true joint distribution of $(X,Y)$ we then compute a p-value as the proportion of observations in the reference distribution that are greater or equal to the value of the U-statistic when computed on $\cD$. 

This procedure differs from a standard permutation test only in that the reference distribution, and hence critical value for rejection, differ slightly. Empirical tests using small sample sizes suggest, however, that results using the above procedure generalise well to those when using a true permutation test. Computing $U_D$ is sufficiently fast that we do not need to use the above procedure and instead use a standard permutation test.

For comparison, we also compute the power of the permutation test based on the distance covariance $d_{cov}$ as computed by the \texttt{energy} package in R \citep{RizzoSzekely16}. In each simulations, we estimate power using 1000 sample data sets from the relevant joint distribution.

We consider two cases in which we generate samples of size 50 from jointly continuous distributions. First, we let $r = 2$, $X_1,X_2$ be independent samples from a $N(0,1)$ distribution, and $Y = X_1 X_2 + \epsilon$ where $\epsilon \sim N(0,\sigma^2)$, with $\sigma \in \{0, \dots, 5\}$.  Figure \ref{fig:cont-power-prod} depicts the power of the tests of the hypothesis that  $X\indep Y$. For comparison, we have also included the power of the distance covariance when $Y$ has been monotonically transformed by the function $f(y) = \sign(y) \log(|y| + 10)$, this transformation substantially reduces the power of the distance covariance while it would have no impact on the power of the other tests as they are nonparametric. In the second case we let $X_1,X_2$ and $\epsilon$ be as above but define $Y = \exp(-(X_1-X_2)^2) + \epsilon$. Figure \ref{fig:cont-power-exp} displays the power of the tests as we let $\sigma$ vary in $\{0,\dots,2\}$. The power of the test based on $\tau^*_J$ is always near the nominal 0.05 level, this suggests that $\tau^*_J(X,Y)=0$ and thus that $\tau^*_J$ is not D-consistent in this case.

We also consider two cases in which $(X,Y)$ is generated from a jointly discrete distribution. Unlike in the continuous case, the sample size $n$ will vary with $n\in\{16, 20, \dots., 48\}$. Firstly, we let $r = 2$, $X_1,X_2$ be independent samples from a Bernoulli($1/2$) distribution, and $Y = \text{XOR}(X_1, X_2)$. We compute the power of our tests for various sample sizes and plot the results in Figure \ref{fig:dis-power-xor3}. As we would expect from Example \ref{exmp:tau-j-inconsistent}, we see that the power of the test based on $\tau^*_J$ equals 0 at all sample sizes. Secondly, we let $r=3$, $X_1,X_2,X_3$ be independent samples from a Bernoulli($1/2$) distribution, and define $Y = \text{XOR}(X_1, X_2, X_3)$. Figure \ref{fig:dis-power-xor4} displays the power of the tests in this setting. Unlike in the prior case, the power of the test based on $\tau^*_J$ is quite high, again recall from Example \ref{exmp:tau-j-inconsistent} that this is because $r$ is odd. 

We conclude with a mixed case, where the covariates $X_1,X_2$ are continuous but the response, $Y$, is binary. In particular, we let $X_1,X_2$ be independent $N(0,1)$ while $Y\sim \text{Bernoulli}(\text{expit}(6 \ \sin(X_1 X_2))$. Our empirical power computations are displayed in Figure \ref{fig:mixed-power}.

As one may expect, we do not see any one particular independence test dominating the others in our simulations. The fact that the nonparametric tests often perform nearly as well, or better, than the distance covariance is somewhat surprising however as they are invariant to such a wide range of transformations. While it is certainly not a proof, the fact that the tests based on $\tau^*_P$ have power beyond the nominal level in all cases suggests that, unlike $\tau^*_J$, perhaps $\tau^*_P$ is indeed D-consistent.

\begin{figure}
  \centering
  \begin{subfigure}[t]{0.49\textwidth}
    \centering
    \includegraphics[width=\textwidth]{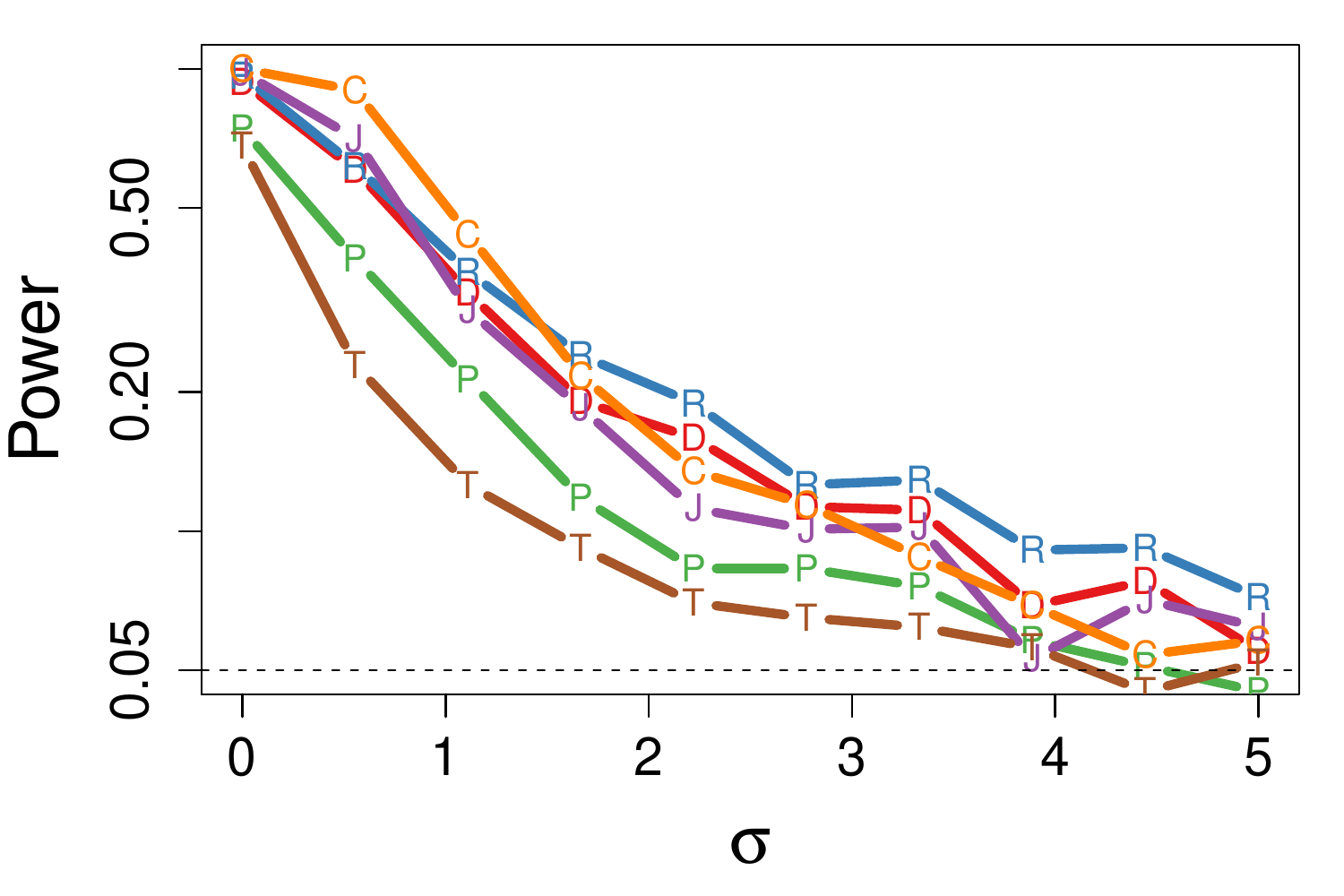}
    \caption{$Y = X_1 X_2 + \epsilon$}\label{fig:cont-power-prod}
  \end{subfigure}
  \begin{subfigure}[t]{0.49\textwidth}
    \centering
    \includegraphics[width=\textwidth]{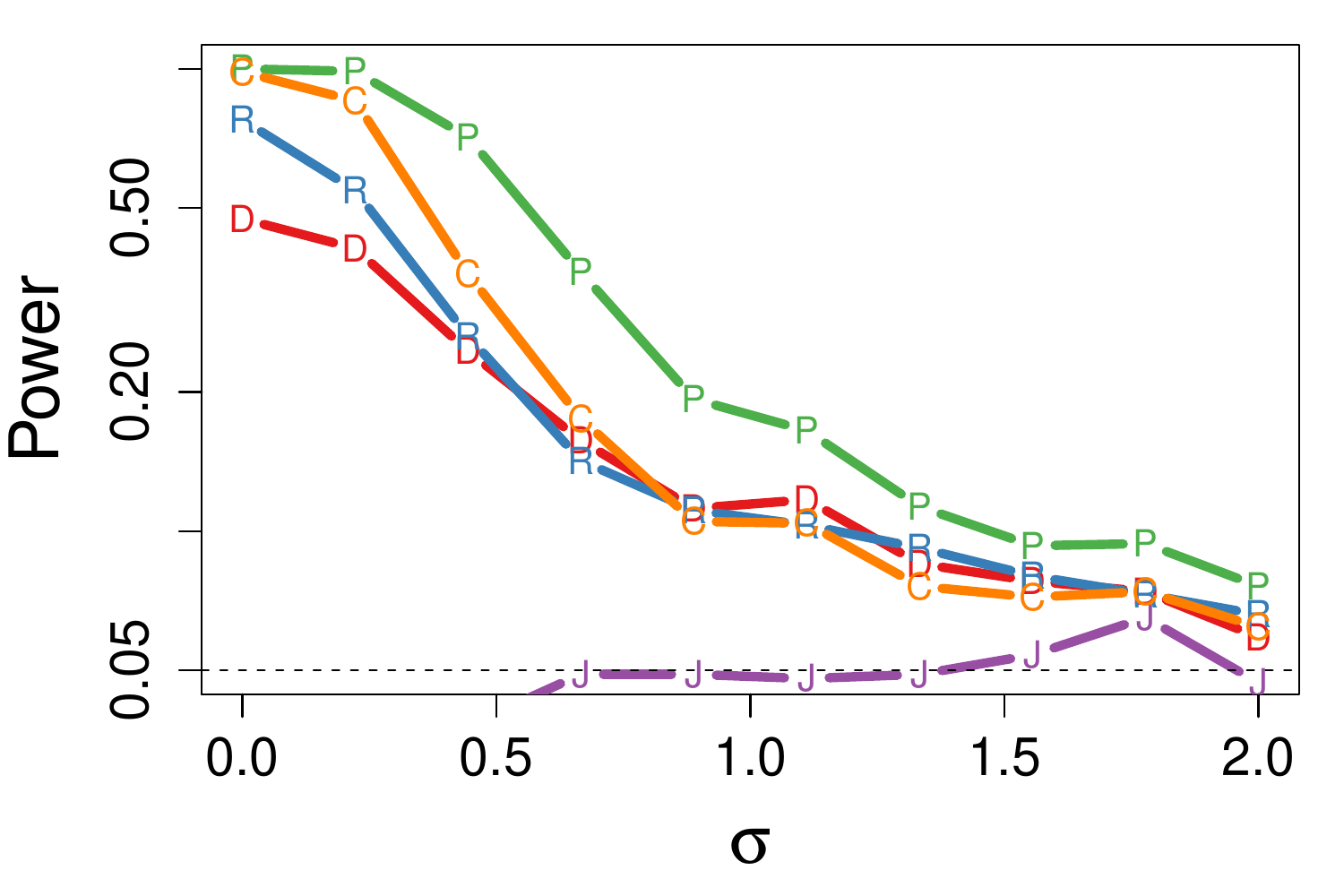}
    \caption{$Y = \exp(-(X_1- X_2)^2) + \epsilon$}\label{fig:cont-power-exp}
  \end{subfigure}

  \caption{Empirical power estimates of permutation tests using $U_D$ (red line with symbol D), $U_{R}$ (blue, R), $U_{\tau^*_P}$ (green, P), $U_{\tau^*_J}$ (purple, J), $d_{cov}$ (orange, C), in the continuous case. The dotted line shows the nominal 0.05 level. For the lines to be visually distinguishable the y-axis is plotted on a log-scale. Here $\sigma$ is the standard deviation of the additive noise $\epsilon$. For Figure \ref{fig:cont-power-prod}, the brown line, with symbol T, corresponds to $d_{cov}$ after applying the strictly increasing transformation $y\mapsto \sign(y)\log(|y|+10)$ to $Y$,  the power of the permutation test based on $d_{cov}$ is substantially harmed by this transformation while the other tests, by monotonic invariance, would not be affected.}
  \label{fig:cont-power} 
\end{figure}

\begin{figure}
  \centering
  \begin{subfigure}[t]{0.49\textwidth}
    \centering
    \includegraphics[width=\textwidth]{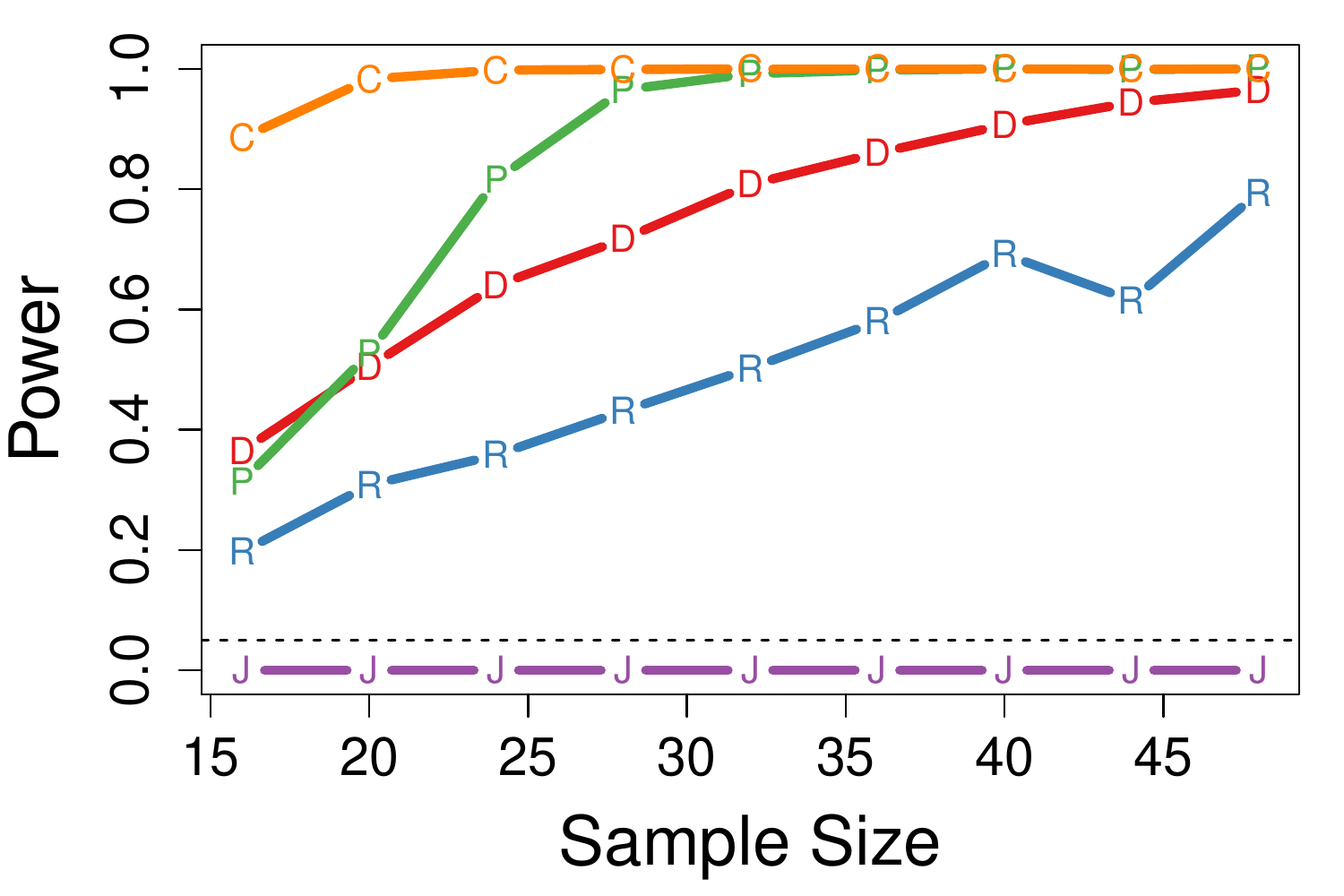}
    \caption{$Y = \text{XOR}(X_1, X_2)$}  \label{fig:dis-power-xor3}
  \end{subfigure}
  \begin{subfigure}[t]{0.49\textwidth} 
    \centering
    \includegraphics[width=\textwidth]{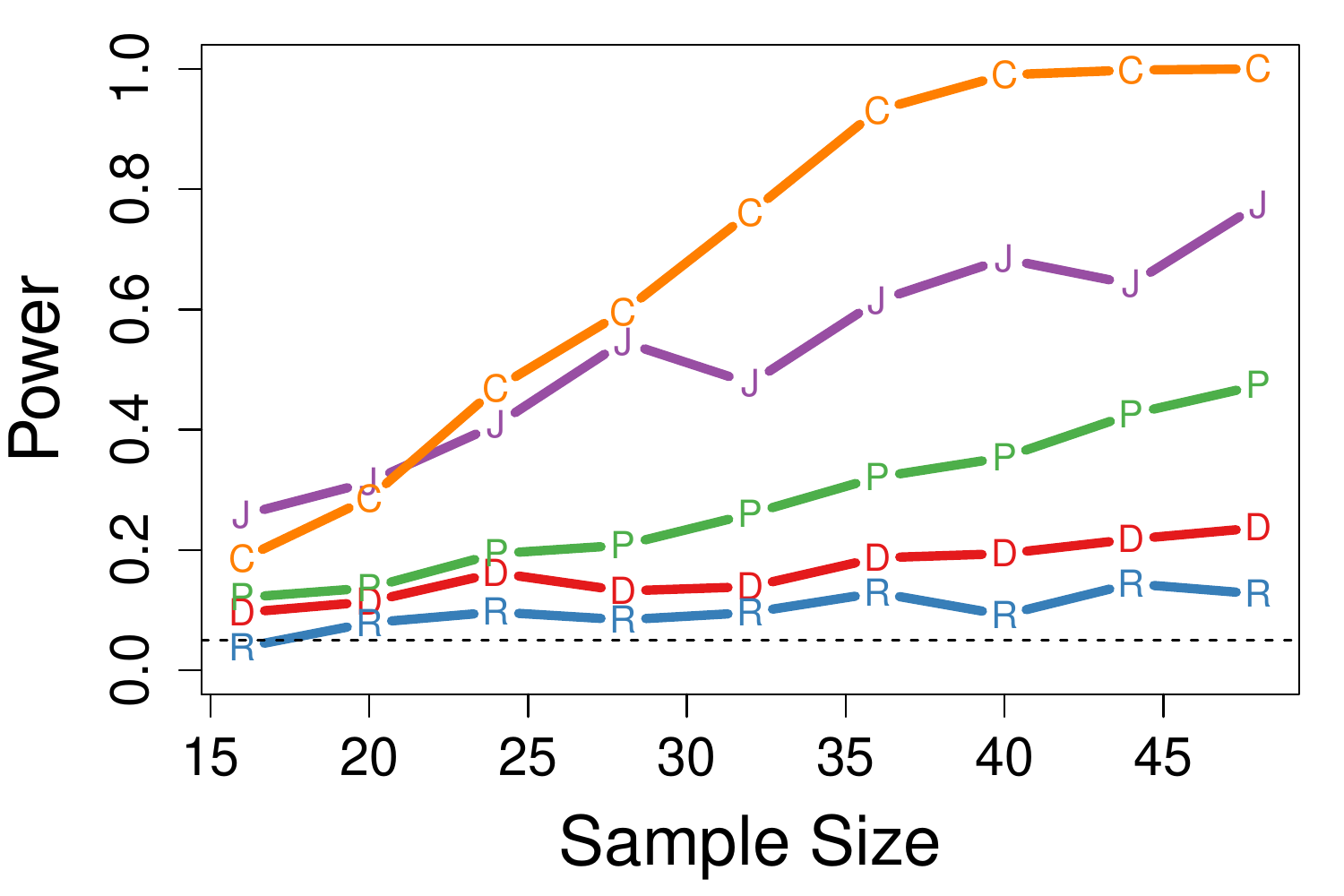}
    \caption{$Y = \text{XOR}(X_1, X_2, X_3)$} \label{fig:dis-power-xor4}
  \end{subfigure}
  
  \caption{Empirical power estimates of permutation tests of independence for a jointly discrete distribution when varying $n\in\{16, 20, 24,\dots, 48\}$. See Figure \ref{fig:cont-power} for the correspondence between line color and test.}
  \label{fig:dis-power} 
\end{figure}

\begin{figure}
        \centering
        \includegraphics[width=.5\textwidth]{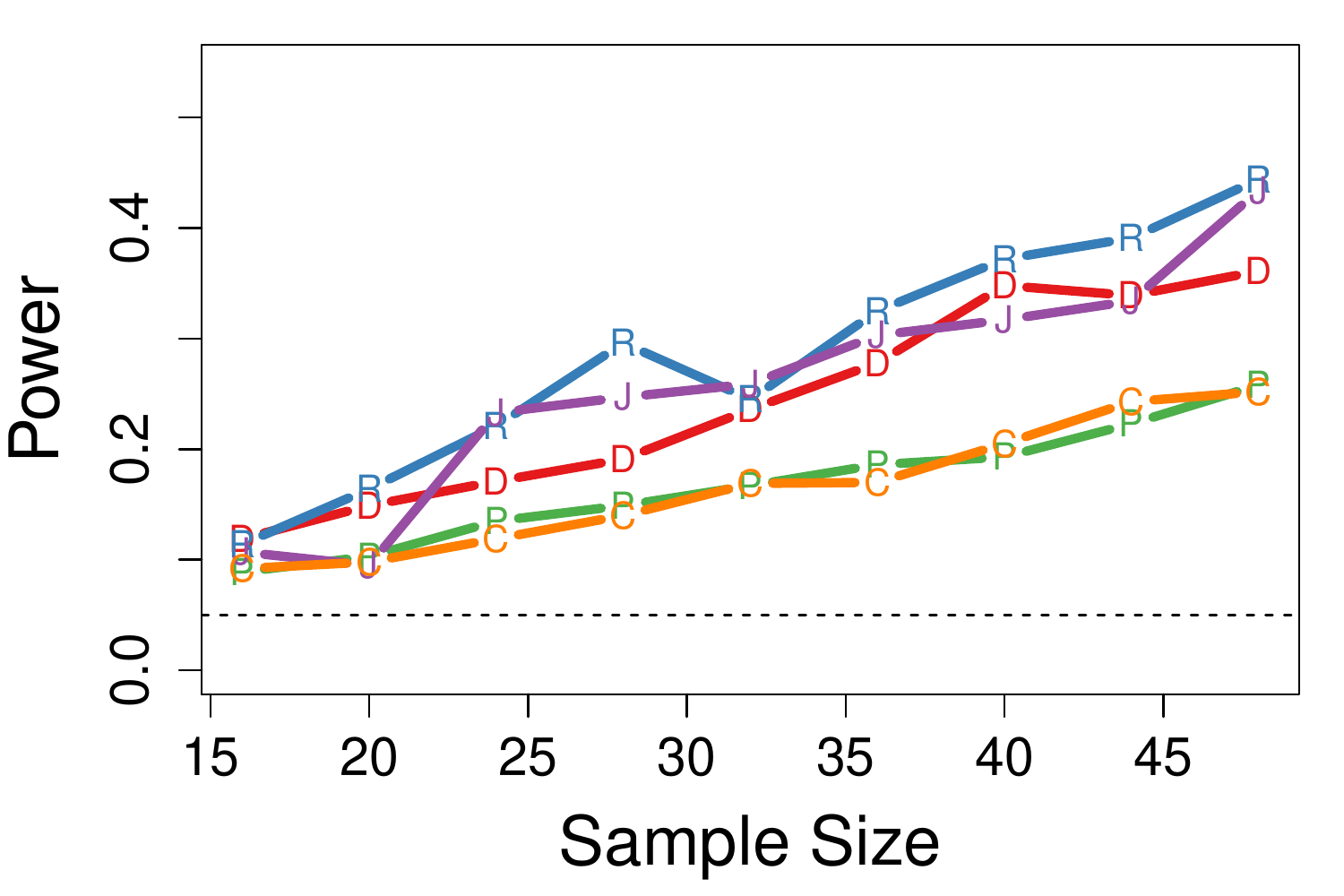}
        \caption{Empirical power estimates of permutation tests of independence when varying $n\in\{16, 20, 24,\dots, 48\}$. Here $Y\sim \text{Bernoulli}(\text{expit}(6 \ \sin(X_1 X_2))$ so the joint distribution $(X,Y)$ is neither jointly continuous or jointly discrete. See Figure \ref{fig:cont-power} for the correspondence between line color and test.}
        \label{fig:mixed-power} 
\end{figure}

\subsection{Computational Efficiency}

While the use of orthogonal range query data structures reduces the asymptotic complexity of computing our U-statistics of interest, such results give little guidance on which algorithms should be used for realistic sample sizes. With practical use in mind, we empirically compare the computational complexity of computing $U_D,U_R,U_{\tau^*_J},$ and $U_{\tau^*_P}$. For these simulations we will generate data from two different distributions; for the first, we let $(X,Y)\sim N_2(0,I_2)$ while, for the second, we let $(X,Y) = (X^1,X^2,Y) \sim N_3(0,I_3)$. We consider the following experiments.

We compute $U_D$ using Equation \eqref{eq:D-u-stat-for-computation} where counts are either computed with a range-tree or looping through the data set. The asymptotic run-time of the range-tree method is $O(n\log_2(n)^{d-1})$ with the more na\"ive method taking $O(n^2)$ time. Both of the above methods are substantially faster than the truly na\"ive $O(n^5)$ strategy of directly computing the sum in Equation \eqref{eq:u-statistic}.

We compute $U_R$ using Equation \eqref{eq:R-u-stat-for-computation} where counts are either computed with an orthogonal range tensor or by looping through the data set. The asymptotic run-time of the orthogonal range tensor method is $O(n^{d})$ while the na\"ive method takes $O(n^{d+1})$ time. As above, both of these methods are much faster than the truly na\"ive $O(n^{d+4})$ strategy.

We compute both $U_{\tau^*_J}$ and $U_{\tau^*_P}$ using our range-tree methods and by definition. The range-tree methods require $O(n^2\log_2(n)^{2d-1})$ time while the na\"ive methods take $O(n^4)$ time.

The results of the above computations are available in Figure \ref{fig:u-time}. From the asymptotic analysis, one would expect that the benefits of using our efficient range query data structures would diminish in higher dimensions and, indeed, that is exactly what the figures show. Comparing Figures \ref{fig:u-d-2-time} and \ref{fig:u-d-3-time}, for instance, we see that when $(X,Y)\sim N_{2}(0,I_2)$ the na\"ive algorithm performs worse than the other for almost all sample sizes but, when moving up to dimension 3 with $(X,Y)\sim N_{3}(0,I_3)$, it is only for sample sizes greater than $\sim$3000 that the range-tree method out-performs the na\"ive strategy.

As Figure \ref{fig:u-p-3-time} shows, computing $U_{\tau^*_P}$ using range-trees is, for reasonable sample sizes, substantially slower than computing $U_{\tau^*_P}$ by definition. This is not surprising considering the many large constant factors that are hidden in the asymptotic analysis.

\begin{figure}
  \centering
  \begin{subfigure}[t]{0.49\textwidth}
    \centering 
    \includegraphics[width=\textwidth]{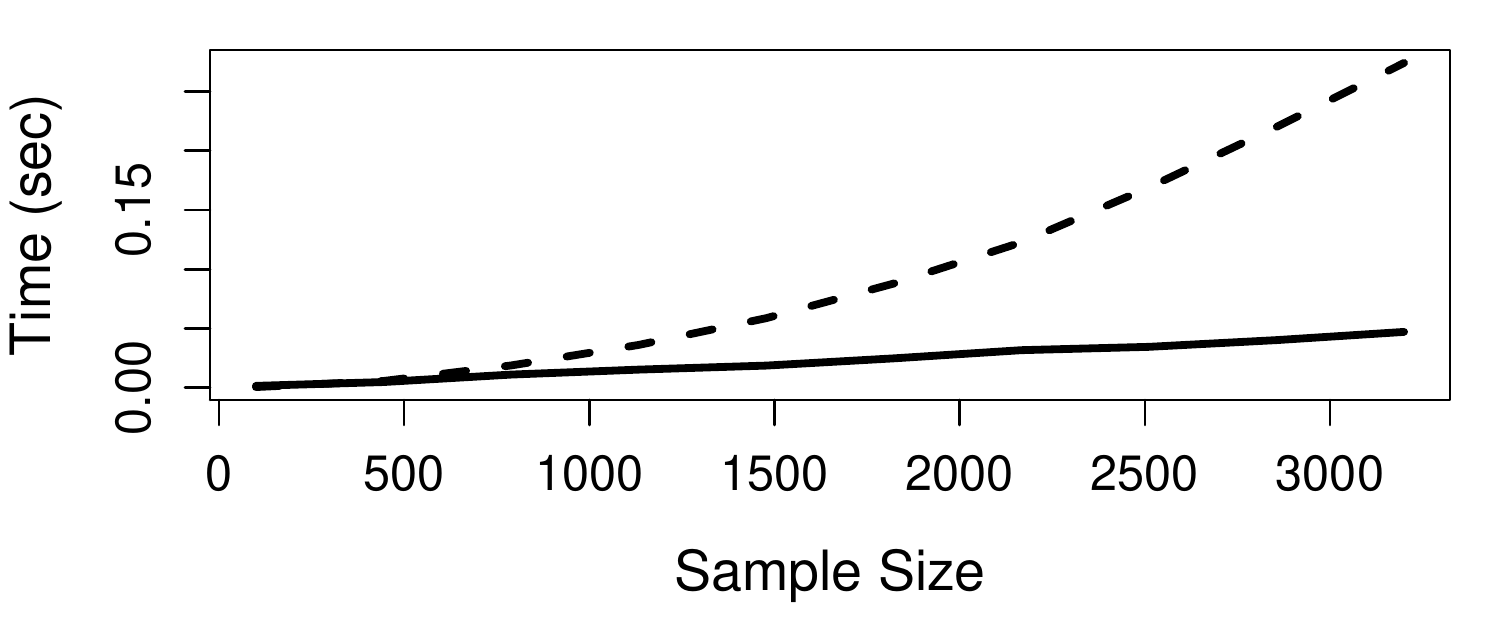}
    \caption{$U_{D},\ (X,Y)\sim N_2(0, I_2)$} \label{fig:u-d-2-time}
  \end{subfigure}
  \begin{subfigure}[t]{0.49\textwidth}
    \centering
    \includegraphics[width=\textwidth]{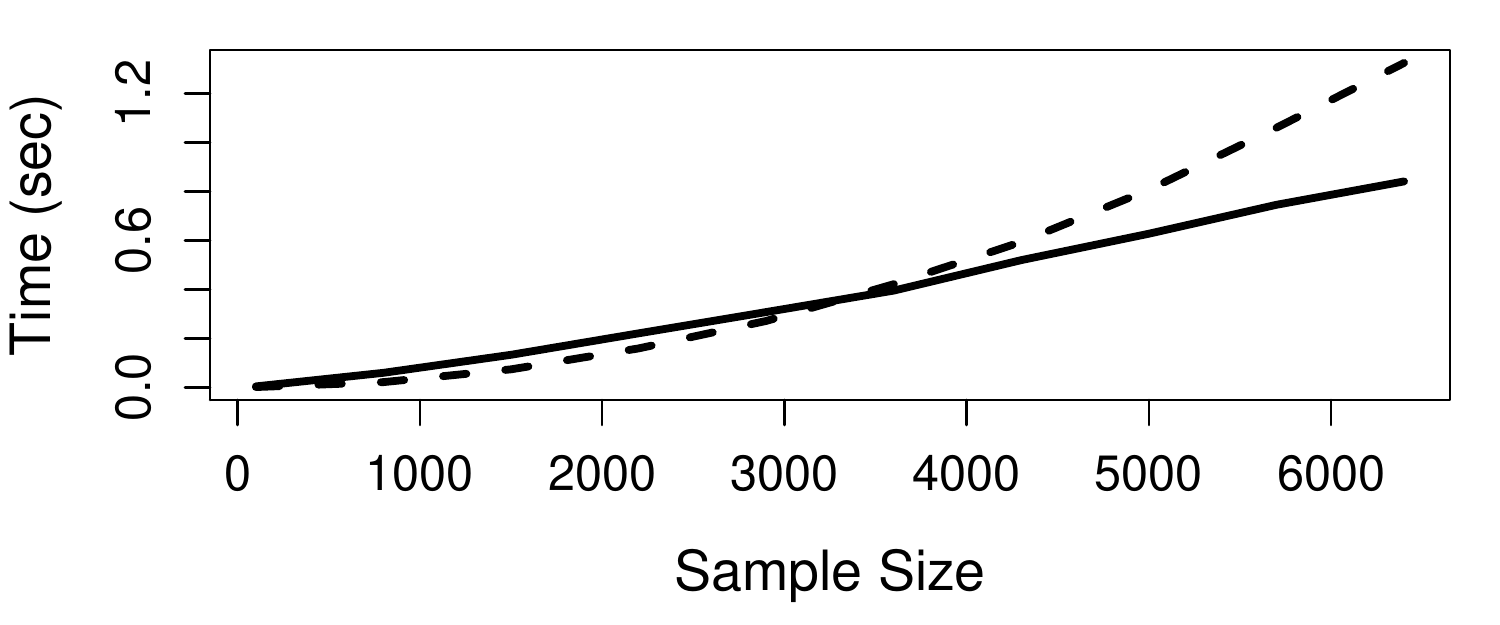}
    \caption{$U_{D},\ (X,Y)\sim N_3(0, I_3)$} \label{fig:u-d-3-time}
  \end{subfigure}
  
  \begin{subfigure}[t]{0.49\textwidth}
    \centering
    \includegraphics[width=\textwidth]{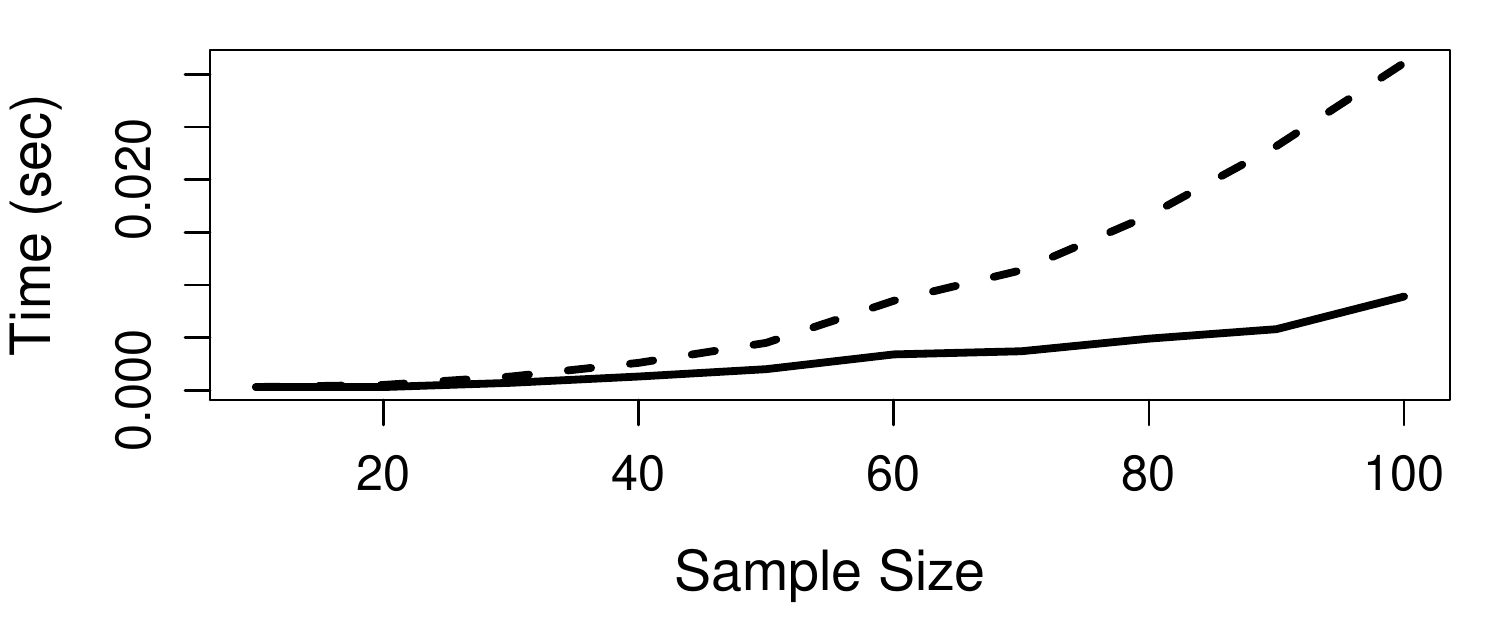}
    \caption{$U_{R},\ (X,Y)\sim N_2(0, I_2)$} \label{fig:u-r-2-time}
  \end{subfigure}
  \begin{subfigure}[t]{0.49\textwidth}
    \centering
    \includegraphics[width=\textwidth]{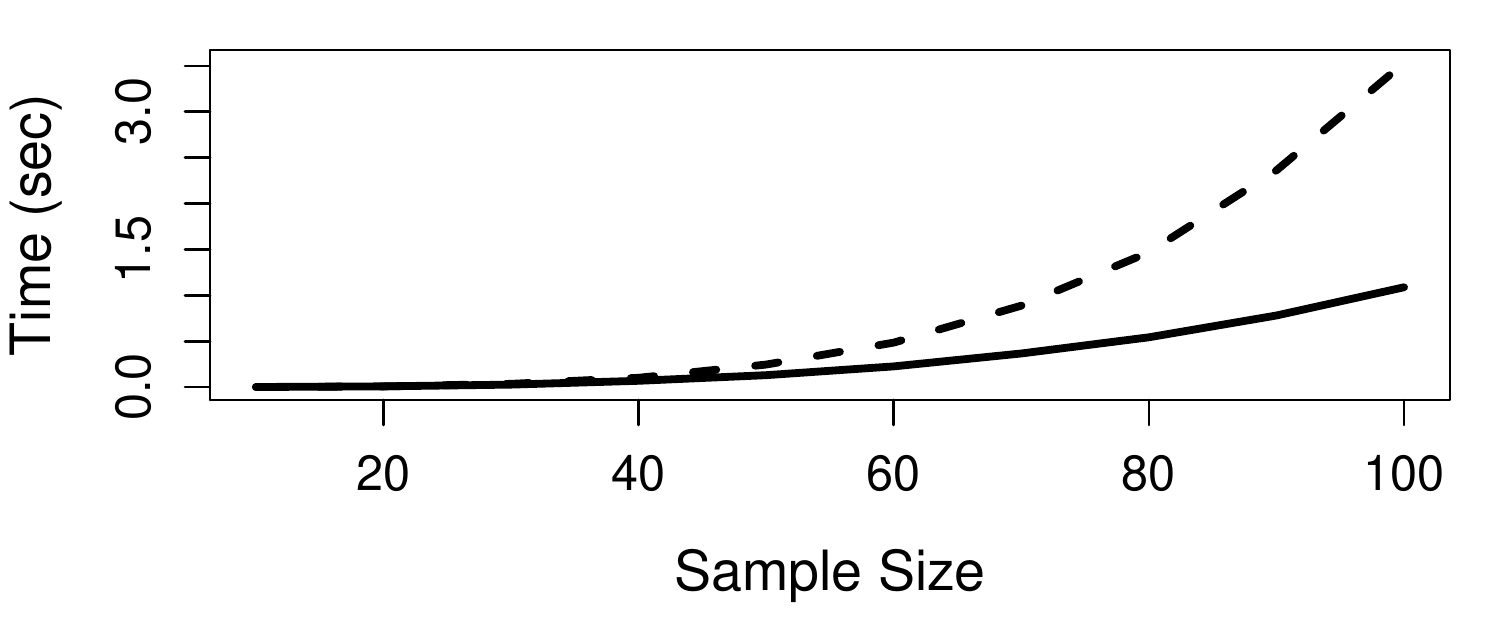}
    \caption{$U_{R},\ (X,Y)\sim N_3(0, I_3)$} \label{fig:u-r-3-time}
  \end{subfigure}
  
  \begin{subfigure}[t]{0.49\textwidth}
    \centering
    \includegraphics[width=\textwidth]{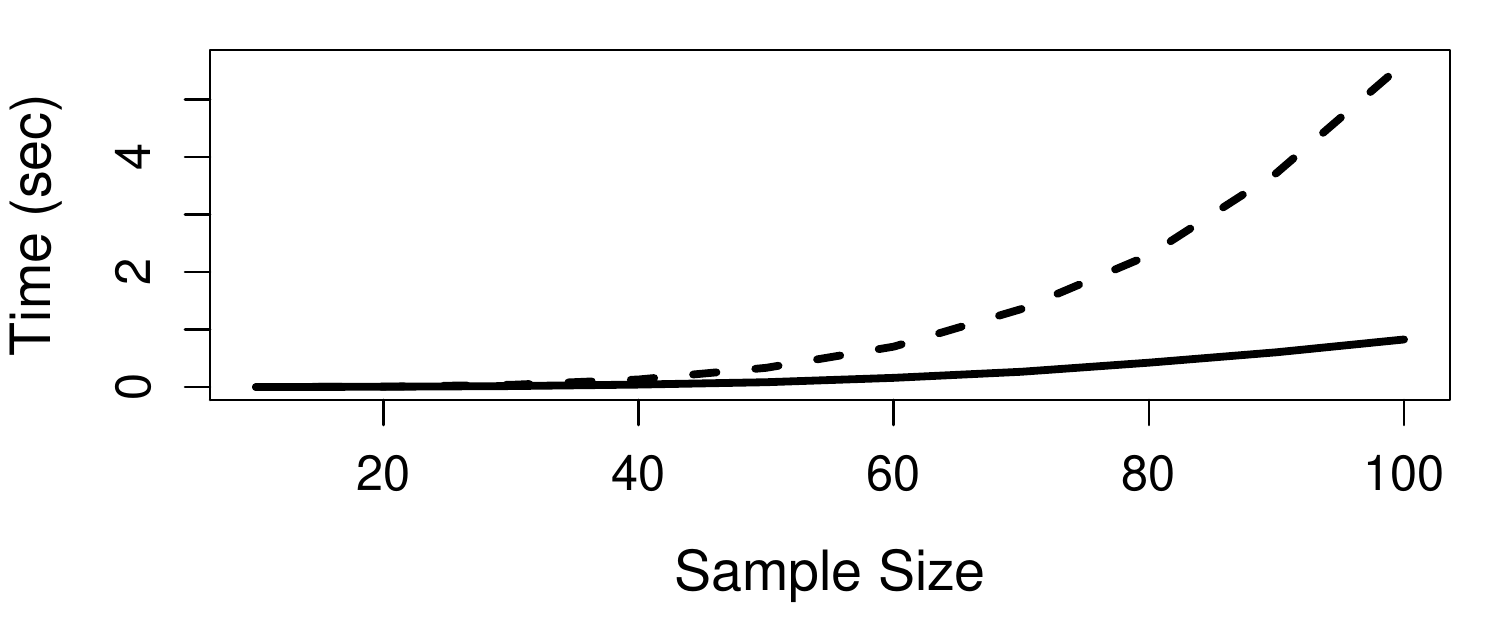}
    \caption{$U_{\tau^*_J},\ (X,Y)\sim N_3(0, I_3)$} \label{fig:u-j-3-time}
  \end{subfigure}
  \begin{subfigure}[t]{0.49\textwidth}
    \centering
    \includegraphics[width=\textwidth]{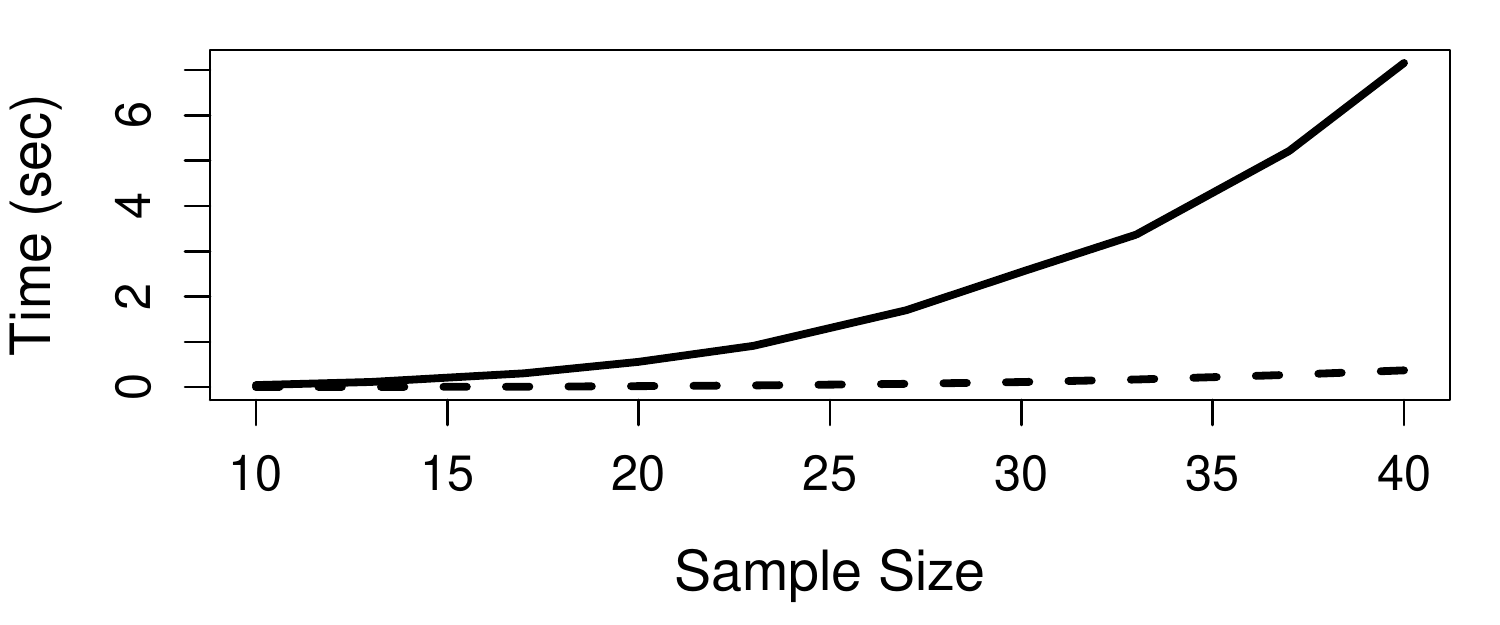}
    \caption{$U_{\tau^*_P},\ (X,Y)\sim N_3(0, I_3)$} \label{fig:u-p-3-time}
  \end{subfigure}

  \caption{The computation time of our U-statistics at various sample sizes comparing the benefits of using (solid lines), and not using (dashed lines), efficient data structures for orthogonal range queries. The na\"ive methods are substantially slower except for two cases, for sample sizes less than $\approx$3000 in (b) and for all tested sample sizes in (f). }
  \label{fig:u-time} 
\end{figure}

\appendix

\section{Asymptotic Theory of U-Statistics}

This section gives a brief review of the asymptotic theory of
U-statistics we require, the book of \cite{Serfling:1980} provides an
in-depth introduction to the topic for the interested reader. Let
$Z^1,Z^2,\dots$ be independent and identically distributed random
vectors taking their values in $\bR^d$ with $d\geq 1$. We say a
function $\kappa:\bR^{d\times m} \to \bR$ is a \emph{symmetric kernel
  function} if its value is invariant to any permutation of its $m$
arguments. Given a symmetric kernel function $\kappa$, we call
\begin{align}
  \label{eq:Un}
  U_n = \frac{1}{{n \choose m}} \sum_{(i_1,\dots,i_m)\in C(n,m)} \kappa(Z_{i_1},\dots,Z_{i_m}),
\end{align}
the \emph{U-statistic with kernel $\kappa$}.  Here, $C(n,m) =
\{(i_1,\dots,i_m)\in \{1,\dots,n\}^m: i_1 < i_2 < \dots < i_m\}$.
Clearly, $EU_n = E\kappa(Z^1,\dots,Z^m)$.

The asymptotics of U-statistics rely deeply on the functions
\begin{equation} \label{eq:k_i}
  \kappa_i(z_1,\dots,z_i) =  E[\kappa(z_1,\dots,z_i, Z_{i+1},\dots,Z_{m})], \quad \text{for $i =1,\dots,m$},
\end{equation}
and their variances
\begin{equation}\label{eq:sigma_i^2}
  \sigma_i^2 = \Var[\kappa_i(Z_1,\dots,Z_i)], \quad \text{for $i =1,\dots,m$}.
\end{equation}
It is well known that $\sigma_1^2\leq \sigma_2^2\leq \dots\leq \sigma_m^2$.

\begin{theorem}[\citet{Serfling:1980}] \label{theorem: normal asymptotics} 
  If the kernel $\kappa$ of the statistic $U_n$ from~(\ref{eq:Un}) has variance $\sigma_m^2 < \infty$, then
  \begin{align*}
    \sqrt{n}(U_n-\theta) \to N(0, m^2 \sigma_1^2)
  \end{align*}
  in distribution.
\end{theorem}

If $\sigma_1^2 = 0$ then the above asymptotic Gaussian distribution is
degenerate and $\sqrt{n}(U_n-\theta) \overset{p}{\to} 0$.  If
$\sigma_1^2 = 0$ and $\sigma_2^2 >0$, then $U_n$ is a \emph{degenerate
  of order 2} and one obtains a non-degenerate limiting
distribution by scaling $U_n$ by a factor of $n$.   In this case, the
limiting distribution is determined by the eigenvalues
of the operator $A_\kappa$ which maps a square integrable function
$g:\bR^d\to\bR$ to the function
$z\mapsto E[(\kappa_2(z, Z_1) - \theta)g(Z_1)]$. 

\begin{theorem}[\citet{Serfling:1980}] \label{theorem:
    degenerate asymptotics} If the kernel $\kappa$ of the statistic
  $U_n$ from~(\ref{eq:Un}) has variance $\sigma_m^2 < \infty$ and
  $\sigma_1^2 = 0 < \sigma_2^2$, then
  \begin{align*}
    n(U_n - \theta) \to {m \choose 2}\sum_{i=1}^\infty \lambda_i (\chi_{1i}^2 - 1)
  \end{align*}
  in distribution where $\chi_{11}^2,\chi_{12}^2,\dots$ are independent and identically distributed random variables
  that follow a chi-square distribution with 1 degree of freedom, and
  the $\lambda_i$'s are the eigenvalues, taken with multiplicity,
  associated with a system of orthonormal eigenfunctions of the
  operator $A_\kappa$.
\end{theorem}

\section{Proofs} \label{app:proofs}

\subsection{Proofs for Section \ref{sec:preliminaries}} \label{app:proofs-bivariate}

\begin{proofof}[Theorem \ref{thm:prod-consistent}]
  If $X\indep Y$, then $R(X,Y) = 0$ because $F_{XY}= F_XF_Y$.

  Now suppose that $X\cancel\indep Y$. Then, by the definition of independence, there exist $(x,y)\in\bR^{r+s}$ such that $F_{XY}(x,y) \not= F_X(x)F_Y(y)$. Since $F_{XY}(x,y) \leq \min(F_X(x), F_Y(y))$, $F_{XY}(x,y) \not= F_X(x)F_Y(y)$ implies that $F_{X}(x),F_Y(y) > 0$.

  We now define $\tilde{x}\in\bR^r, \tilde{y}\in \bR^s$ as follows. Let
  \begin{align*}
    \tilde{x}_i &= \arg\min\{x^*\mid x^* \leq x_i \text{ and } F_{X_i}(x^*) = F_{X_i}(x_i)\}  \quad (i=1,\dots,r) \quad \text{ and} \\
    \tilde{y}_i &= \arg\min\{y^*\mid y^*\leq y_i \text{ and } F_{Y_i}(y^*) = F_{Y_i}(y_i)\}  \hspace{6mm} (i=1,\dots,s).
  \end{align*}
  By the right continuity of cumulative distribution functions, we have that such $\tilde{x}_i,\tilde{y}_i$ exist and that $F_{X_i}(\tilde{x}_i) = F_{X_i}(x_i)$ and $F_{Y_j}(\tilde{y}_j) = F_{Y_j}(y_j)$ for all $i$ and $j$. We will now show that $F_{XY}(\tilde{x},\tilde{y}) = F_{X,Y}(x,y)$. Clearly $F_{XY}(\tilde{x},\tilde{y}) \leq F_{X,Y}(x,y)$. Suppose, for contradiction, that $F_{XY}(\tilde{x},\tilde{y}) < F_{X,Y}(x,y)$. Write $\tilde{z} = (\tilde{x},\tilde{y})$ and let $i$ be the smallest index for which $F_{XY}(\tilde{z}_1,\dots,\tilde{z}_{i},z_{i+1},\dots,z_{r+s}) < F_{X,Y}(z)$, by assumption such an $i$ exists. Without loss of generality assume that $i \leq r$. Then we have that
  \begin{align*}
    F_{X_i}(x_i) &= P(X_i \leq x_i) \\
                 &= F_{XY}(\tilde{x}_1,\dots,\tilde{x}_{i-1},x_i,x_{i+1},\dots,x_r,y) + P(X_i \leq x_i \text{ and } (X_1 > \tilde{x_1} \text{ or } \dots \text{ or } Y_s > y_s))).
  \end{align*}
  Now clearly both $F_{XY}(\tilde{x}_1,\dots,\tilde{x}_{i-1},x_i,x_{i+1},\dots,x_r)$ and $P(X_i \leq x_i \text{ and } (X_1 > \tilde{x_1} \text{ or } \dots \text{ or } Y_s > y_s)))$ are non-decreasing in $x_i$ and thus, since $F_{XY}(\tilde{z}_1,\dots,\tilde{z}_{i},z_{i+1},\dots,z_{r+s}) < F_{X,Y}(z)$, we have that
  \begin{align*}
    F_{X_i}(x_i) &> F_{XY}(\tilde{x}_1,\dots,\tilde{x}_{i-1},\tilde{x}_i,x_{i+1},\dots,x_r) + P(X_i \leq \tilde{x}_i \text{ and } (X_1 > \tilde{x_1} \text{ or } \dots \text{ or } Y_s > y_s))) \\
                 &= F_{X_i}(\tilde{x}_i).
  \end{align*}
  But this contradicts what we have shown above, that $F_{X_i}(x_i) = F_{X_i}(\tilde{x}_i)$. It follows that $F_{XY}(\tilde{x},\tilde{y}) = F_{XY}(x,y)$ as claimed. An essentially identical argument to the one above also shows that $F_{X}(\tilde{x}) = F_X(x)$ and $F_Y(\tilde{y}) = F_Y(y)$. Hence we have that
  \begin{align*}
    F_{XY}(\tilde{x},\tilde{y}) - F_X(\tilde{x})F_Y(\tilde{y}) = F_{XY}(x,y) - F_X(x)F_Y(y) \not= 0.
  \end{align*}

  Now let $\cI_X = \{i\in [r] \mid \lim_{x^*\to \tilde{x}_i -}F_{X_i}(x^*) \not= F_{X_i}(\tilde{x}_i)\}$ so that for all $i\in\cI_X$ we have that $F_{X_i}$ has a jump discontinuity at $\tilde{x}_i$ and thus $P(X_i = \tilde{x}_i) > 0$. Let $\cI_Y$ be the corresponding set of such indices for the $F_{Y_i}$. Now, for any $i\in [r]$ and $\delta>0$ define $B^i_{\delta} = [\tilde{x}_i-\delta,\tilde{x}_i]$ if $i\not\in \cI_X$ and $B^i_{\delta}=\{\tilde{x}_i\}$ if $i\in\cI_X$. By our definition of $\tilde{x}_i$ and $\cI_X$ we have that $P(X_i\in B^i_{\delta}) > 0$. Similarly define, for any $i\in [s]$ and $\delta>0$, $C^i_{\delta} = [\tilde{y}_i-\delta,\tilde{y}_i]$ if $i\not\in \cI_Y$ and $C^i_{\delta} = \{\tilde{y}_i\}$ if otherwise. Again we have that $P(Y_i\in C^i_{\delta}) > 0$. Let $B_\delta = B^1_{\delta}\times \dots \times B^r_{\delta}$ and $C_\delta = C^1_{\delta}\times \dots \times C^s_{\delta}$.

  Claim: there exists $\delta>0$ such that for all
  \begin{align*}
    (x,y) \in D_\delta = B_{\delta} \times C_{\delta}
  \end{align*}
  we have $F_{XY}(x,y) - F_X(x)F_Y(y) \not= 0$.

  If this claim is true we have that
  \begin{align*}
    R(X,Y)&=\int_{\bR^{r+s}}(F_{XY}(x,y)-F_X(x)F_Y(y))^2\prod_{i=1}^r \text{d}F_{X_i}(x_i)\prod_{j=1}^s\dF_{Y_j}(y_j)\\
    &\geq \int_{D_{\delta}}(F_{XY}(x,y)-F_X(x)F_Y(y))^2\prod_{i=1}^r \text{d}F_{X_i}(x_i)\prod_{j=1}^s\dF_{Y_j}(y_j)\\
    &>0
  \end{align*}
  where the last inequality follows since $D_{\delta}$ has positive measure under $\text{d}F_{X_i}(x_i)\prod_{j=1}^s\dF_{Y_j}(y_j)$ and $(F_{XY}(x,y)-F_X(x)F_Y(y))^2$ is strictly positive for $(x,y)\in D_{\delta}$.

  We now prove the claim. First let $\phi:[0,\infty]\to \bR^r$ be
  defined such that, for all $t\in[0,\infty]$, $\phi(t)_i =
  \tilde{x}_i$ if $i\in \cI_X$ and $\phi(t)_i = t\tilde{x}_i$ if
  $i\not\in\cI_X$. Each $\phi(t)_i$ is non-decreasing in
  $t$. Next consider the function $G:[0,\infty]\to [0,1]$ defined such
  that $G(t) = F_X(\phi(t))$. $G$ is monotone non-decreasing
  and so has only jump discontinuities. If $G$ does not have a jump discontinuity at $t=1$ then for any $\epsilon>0$ we may pick $\delta<1$ sufficiently small that
  \begin{align*}
    F_X(\tilde{x}) - F_X(\phi(s)) = |F_X(\tilde{x}) - F_X(\phi(s))| = |G(1)-G(s)| < \epsilon
  \end{align*}
  for all $1-\delta \leq s\leq 1$. But for any $x \in B_{\delta}$ we have $\phi(1-\delta) \preceq x\preceq \tilde{x}$ and thus
  \begin{align*}
    |F_X(\tilde{x}) - F_X(x)| = F_X(\tilde{x}) - F_X(x) = F_X(\tilde{x}) - F_X((1-\delta)\tilde{x}) = G(1)-G(1-\delta) < \epsilon.
  \end{align*}
  Now suppose otherwise that $G(t)$ has a jump discontinuity at $t=1$. That is, we have $F_X(\tilde{x})-\lim_{t\to1^-}G(t) = a > 0$. Then, by the monotone convergence theorem, 
  \begin{align*}
    a &= F_X(\tilde{x})-\lim_{t\to1^-}G(t) \\
      &= P(\bigwedge_{i=1}^r X_i\leq \tilde{x}_i) - P(\bigwedge_{i\in \cI_X}(X_i \leq \tilde{x}_i) \wedge \bigwedge_{i\not\in \cI_X}(X_i < \tilde{x}_i)) \\
      &= P(\bigwedge_{i\in \cI_X}(X_i \leq \tilde{x}_i) \wedge \bigvee_{i\not\in \cI_X}(X = \tilde{x}_i)) \\
      &\leq P(\bigvee_{i\not\in \cI_X} X_i = \tilde{x}_i) \\
      &\leq \sum_{i\not\in \cI_X} P(X_i = \tilde{x}_i).
  \end{align*}
  But by definition of $\cI_X$ we have that $P(X_i = \tilde{x}_i) =0$ for all $i\in \cI_X$, it thus follows that $0<a\leq 0$ a contradiction. Hence $G(t)$ does not have a jump discontinuity at $t=1$.

  Similar arguments hold for $F_Y$ and $F_{XY}$ and hence, for any
  $\epsilon >0$ there exists some $\delta >0$ such that for any
  $(x,y)\in B_\delta\times C_\delta$ we have $|F(\tilde{x},\tilde{y})
  - F_{XY}(x,y)|, |F(\tilde{x})-F_X(x)|, |F(\tilde{y})-F_Y(y)| <
  \epsilon$. Thus, choosing $\epsilon >0$ sufficient small, clearly
  there exists $\delta$ such that $F_{XY}(x,y)-F_X(x)F_Y(y) \not=0$ for all
  $(x,y)\in B_\delta\times C_\delta$.  This completes the proof as
  noted above.
\end{proofof}

\subsection{Proofs for Section \ref{sec:src}} \label{app:src}

\begin{proofof}[Proposition \ref{prop:standard-measures-are-srcs}]
  \cite{BergsmaDassios14} show that
  \begin{align*}
    \tau^* = E[a(X^{[4]})\ a(Y^{[4]})]
  \end{align*}
  where
  \begin{align*}
    a(w^{[4]}) &= I_{[w^1,w^2<w^3,w^4]} + I_{[w^3,w^4<w^1,w^2]} - I_{[w^1,w^3<w^2,w^4]} - I_{[w^2,w^4<w^1,w^3]}.
  \end{align*}
  This is exactly our claimed result.

  Recall that $\tau$ can be expressed as
  \begin{align*}
    \tau = E[2\ I_{[X^1 < X^2]}\ (I_{[Y^1 < Y^2]} - I_{[Y^2 < Y^1]})].
  \end{align*}
  Lemma \ref{lem:alternate-src-forms} then immediately gives our result for $\tau$. Moreover, letting $\gamma=\nu=\tau$ in the proof of Proposition \ref{prop:general-src-results} and relabeling 2 as 4, and vice versa, we have our claimed form for $\tau^2$.

  We now show that our result for $D$. For any $z=(x,y)\in\bR^{r+s}$ let
  \begin{align}
    c^+(x,y) =&\ F_{XY}(x,y)^2 (1-F_X(x)-F_Y(y)+F_{XY}(x,y))^2, \label{eq:multivariate-D-con-p}\\
    c^-(x,y) =&\ (F_X(x)-F_{XY}(x,y))^2 (F_Y(y)-F_{XY}(x,y))^2, \label{eq:multivariate-D-con-n}\\
    c(x,y) =&\ c^+(x,y) + c^-(x,y), \text{ and} \label{eq:multivariate-D-con}\\
    d(x,y) =&\ \sqrt{c^+(x,y)c^-(x,y)} \label{eq:multivariate-D-dis}\\
    =&\ (F_X(x)-F_{XY}(x,y)) (F_Y(y)-F_{XY}(x,y)) F_{XY}(x,y) (1-F_X(x)-F_Y(y)+F_{XY}(x,y)). \nonumber
  \end{align}
  It is easy to check that $c(x,y) - 2\ d(x,y) = (F_{XY}(x,y) - F_X(x)F_Y(y))^2$ and thus
  \begin{align}\label{eq:d-as-concordant-diff}
    D = \int_{\bR^2} c(x,y) - 2d(x,y) \dF_{XY}(x,y).
  \end{align}
  It is interesting to note that $4(c(x,y)- 2d(x,y)) = \tau^*(1_{[X\leq x]}, 1_{[Y\leq y]})$ so that $D$ can be interpreted as a weighted integral of $\tau^*$ applied to discretised versions of the $Z^i$.  This is the perspective from Section \ref{sec:discretization}. Now
  \begin{align*}
    c^+(x,y) &= E[1_{[X^1,X^2\preceq x]}1_{[X^3,X^4\not\preceq x]}1_{[Y^1,Y^2\preceq y]}1_{[Y^3,Y^4\not\preceq y]}], \\
    c^-(x,y) &= E[1_{[X^1,X^2\preceq x]}1_{[X^3,X^4\not\preceq x]}1_{[Y^3,Y^4\preceq y]}1_{[Y^1,Y^2\not\preceq y]}], \quad \text{ and }\\
    d(x,y) &= E[1_{[X^1,X^2\preceq x]}1_{[X^3,X^4\not\preceq x]} 1_{[Y^1,Y^3\preceq y]} 1_{[Y^2,Y^4\not\preceq y]}] \\
    &= E[1_{[X^1,X^2\preceq x]}1_{[X^3,X^4\not\preceq x]} 1_{[Y^2,Y^4\preceq y]}1_{[Y^1,Y^3\not\preceq y]}].
  \end{align*}
  This gives
  \begin{align*}
    &\int_{\bR^{r+s}}c^+(x,y)\dF_{XY}(x,y) = \int_{\bR^2}E[1_{[X^1,X^2\preceq x]}1_{[X^3,X^4\not\preceq x]}1_{[Y^1,Y^2\preceq y]}1_{[Y^3,Y^4\not\preceq y]}]\dF_{XY}(x,y) \\
    &= E[\int_{\bR^{r+s}}1_{[X^1,X^2\preceq x]}1_{[X^3,X^4\not\preceq x]}1_{[Y^1,Y^2\preceq y]}1_{[Y^3,Y^4\not\preceq y]}\dF_{XY}(x,y)] \\
    &= E[1_{[X^1,X^2\preceq X^5]}1_{[X^3,X^4\not\preceq X^5]}1_{[Y^1,Y^2\preceq Y^5]}1_{[Y^3,Y^4\not\preceq Y^5]}] \\
    &= E[I_D(X^{[5]})I_D(Y^{[5]})]
  \end{align*}
  Similarly one may also show that
  \begin{align*}
    \int_{\bR^2}c^-(x,y)\dF_{XY}(x,y) &= E[I_D(X^{[5]})I_D(Y^{4,3,2,1,5})], \text{ and} \\
    \int_{\bR^2}d(x,y)\dF_{XY}(x,y) &= E[I_D(X^{[5]})I_D(Y^{1,3,2,4,5})] \\
                                 &= E[I_D(X^{[5]})I_D(Y^{4,2,3,1,5})].
  \end{align*}
  From this and Equation \eqref{eq:d-as-concordant-diff} it is easy to see that
  \begin{align*}
    D(X,Y) = E[I_D(X^{[5]})\sum_{\sigma\in H_{\tau^*}}\sign(\sigma)I_D(Y^{[5]})].
  \end{align*}
  Our claim then follows by Lemma \ref{lem:alternate-src-forms}. Following essentially identical steps as for $D(X,Y)$, one may also show our desired result for $R(X,Y)$.
\end{proofof}

\begin{proofof}[Proposition \ref{prop:general-src-results}]
  Without loss of generality let $\mu$ be as in Equation \eqref{eq:src}.

  We first show that $\mu$ is nonparametric. Letting $h_{X,i}:\bR\to\bR$ and $h_{Y,j}:\bR\to\bR$ strictly increasing functions for $i\in[r],\ j\in[s]$ and letting $h_X(x) = (h_{X,1}(x_1),\dots,h_{X,r}(x_r))$ and $h_Y(y) = (h_{Y,1}(y_1),\dots,h_{Y,s}(y_s))$ we wish to show that $\mu(X,Y) = \mu(h_X(X), h_Y(Y))$.

  It is trivial to check that, as the $h_{X,i}$ and $h_{Y,i}$ are strictly increasing we have that $\cR(X^{[m]})=\cR(h_X(X^{[r]}))$ and $\cR(Y^{[m]})=\cR(h_Y(Y^{[s]}))$. Given this, the claim follows immediately as rank indicator functions depend on their inputs only through the joint ranks of the inputs.

Next we show that $\mu$ is I-consistent. Recall that, by definition, $H\subset S_m$ has an equal number of even and odd permutations. It follows then that
  \begin{align*}
    \sum_{\sigma\in H}\sign(\sigma) = 0.
  \end{align*}
  When $X\indep Y$, since the $X^{[m]}$ and $Y^{[m]}$ are independent and identically distributed respectively, we have
  \begin{align*}
    \mu(X,Y) &= E\Big[\Big(\sum_{\sigma \in H} \sign(\sigma)\ I_{X}(X^{\sigma[m]})\Big)\ \Big(\sum_{\sigma \in H}\sign(\sigma)\ I_{Y}(Y^{\sigma[m]})\Big)\Big] \\
             &= E\Big[I_{X}(X^{[m]})\Big]\ E\Big[I_{Y}(Y^{[m]})\Big]\ (\sum_{\sigma \in H} \sign(\sigma))\ (\sum_{\sigma \in H}\sign(\sigma))\\
             &= E\Big[I_{X}(X^{[m]})\Big]\ E\Big[I_{Y}(Y^{[m]})\Big]\cdot 0 \cdot 0\\
             &= 0
  \end{align*}
  which proves the claim.
  
  Finally we show that Symmetric Rank Covariances are closed under products. Without loss of generality assume that $\gamma = \mu_{I_X,I_Y,H}$ and $\nu = \mu_{\tilde{I}_X,\tilde{I}_Y,\tilde{H}}$. We have that
  \begin{align*}
    \mu(X,Y) &= E\Big[\Big(\sum_{\sigma \in H} \sign(\sigma)\ I_{X}(X^{\sigma[m]})\Big)\ \Big(\sum_{\sigma \in H}\sign(\sigma)\ I_{Y}(Y^{\sigma[m]})\Big)\Big],\quad \text{and} \\
    \nu(X,Y) &= E\Big[\Big(\sum_{\sigma \in \tilde{H}} \sign(\sigma)\ \tilde{I}_{X}(X^{\sigma[n]})\Big)\ \Big(\sum_{\sigma \in \tilde{H}}\sign(\sigma)\ \tilde{I}_{Y}(Y^{\sigma[n]})\Big)\Big].
  \end{align*}
  In the following we will implicitly let $I_X(x^{[m+n]}) = I_X(x^{[m]})$, $I_Y(y^{[m+n]}) = I_Y(y^{[m]})$, $\tilde{I}_X(x^{[m+n]}) = \tilde{I}_X(x^{[n]})$, and $\tilde{I}_Y(y^{[m+n]}) = \tilde{I}_Y(y^{[n]})$ so that these indicator functions drop unused inputs. Now let $\gamma\in S_{[m+n]}$ be the permutation that cyclically shifts all elements $n$ units to the right, so that $i\in[m+n]$ is taken to $i + n \mod (m + n)$ by $\gamma$. Then let $g_\gamma:\bR^{d\times (m+n)}\to\bR^{d\times(m+n)}$ be the function which acts on it's input with $\gamma$, that is we let $g_\gamma(w^{[m+n]}) = w^{\gamma[m+n]}$ for all $w^{[m+n]}\in \bR^{d\times (m+n)}$. Then define
  \begin{align*}
    \ol{H} &= \gamma^{-1} \tilde{H} \gamma, \quad
             \ol{I}_X =  \tilde{I}_X\circ g_\gamma, \quad \text{and} \quad
    \ol{I}_Y =  \tilde{I}_Y\circ g_\gamma.
  \end{align*}
  Clearly $\ol{H}$ is a subgroup of $S_{m+n}$ and it is easy to check that
  \begin{align*}
    \nu(X,Y) &= E\Big[\Big(\sum_{\sigma \in \ol{H}} \sign(\sigma)\ \ol{I}_{X}(X^{\sigma[m+n]})\Big)\ \Big(\sum_{\sigma \in \ol{H}}\sign(\sigma)\ \ol{I}_{Y}(Y^{\sigma[m+n]})\Big)\Big].
  \end{align*}
   Now
  \begin{align*}
    A &= \Big(\sum_{\sigma \in H} \sign(\sigma)\ I_{X}(X^{\sigma[m+n]})\Big)\ \Big(\sum_{\sigma \in H}\sign(\sigma)\ I_{Y}(Y^{\sigma[m+n]})\Big) \text{ and } \\
    B &= \Big(\sum_{\sigma \in \ol{H}} \sign(\sigma)\ \ol{I}_{X}(X^{\sigma[m+n]})\Big)\ \Big(\sum_{\sigma \in \ol{H}}\sign(\sigma)\ \ol{I}_{Y}(Y^{\sigma[m+n]})\Big)
  \end{align*}
  depend only on $Z^{[m+n]}$ through the entries $Z^{[m]}$ and $Z^{m+1,\dots,m+n}$ respectively, it thus follows that $A$ and $B$ are independent. Thus, by how we have defined $\ol{I}_X,\ol{I}_Y$ and $\ol{H}$, we have that
  \begin{align*}
    \gamma\ \nu &= E[A]E[B] \\
                    &= E[AB] \\
                    &= \Big(\sum_{\substack{\sigma \in H \\ \ol{\sigma} \in \ol{H}}} \sign(\sigma\ol{\sigma})\ (I_X\ \ol{I}_{X})(X^{(\sigma\ol{\sigma})[m+n]})\Big)\ \Big(\sum_{\substack{\sigma \in H \\ \ol{\sigma} \in \ol{H}}}\sign(\sigma\ol{\sigma}) \ (I_Y\ \ol{I}_{Y})(Y^{(\sigma\ol{\sigma})[m+n]})\Big) \\
                    &= \Big(\sum_{\sigma \in H\tilde{H}} \sign(\sigma)\ (I_X\ \ol{I}_{X})(X^{\sigma[m+n]})\Big)\ \Big(\sum_{\sigma \in H\tilde{H}}\sign(\sigma) \ (I_Y\ \ol{I}_{Y})(Y^{\sigma[m+n]})\Big)
  \end{align*}

  Where the last line follows since that $H\cap \ol{H}=\{e\}$ implies that $|H\ol{H}|=|H||\ol{H}|$ and since elements of $H$ and $\ol{H}$ commute we have that $H\ol{H}$ is a subgroup of $S_{[m+n]}$. The above equality shows that $\gamma \nu$ is a Symmetric Rank Covariance as claimed.
\end{proofof}

\begin{proofof}[Lemma \ref{lem:alternate-src-forms}]
  For any $\gamma\in H$, by relabeling $Z^{[m]}$ as $Z^{\gamma^{-1}[m]}$, we have that
  \begin{align*}
    E\Big[\sign(\gamma)\ &I_X(X^{\gamma[m]})\ \sum_{\sigma\in H}\sign(\sigma)\ I_Y(Y^{\sigma[m]})\Big] \\
                             &= E\Big[I_X(X^{\gamma\gamma^{-1}[m]})\ \sum_{\sigma\in H}\sign(\sigma)\ \sign(\gamma)\ I_Y(Y^{\sigma\gamma^{-1}[m]})\Big] \\
                             &= E\Big[I_X(X^{[m]})\ \sum_{\sigma\in H}\sign(\sigma\gamma^{-1})\ I_Y(Y^{\sigma\gamma^{-1}[m]})\Big] \\
                             &= E\Big[I_X(X^{[m]})\ \sum_{\sigma\in H}\sign(\sigma)\ I_Y(Y^{\sigma[m]})\Big]
  \end{align*}
  where the third equality holds since $\sign(\gamma) = \sign(\gamma^{-1})$ and the fourth equality holds since $\gamma\in H\implies H = H\gamma^{-1}$. Now plugging the above equality into our definition of $\mu(X,Y)$ gives Equation \eqref{eq:mu-with-x}. By symmetry we obtain Equation \eqref{eq:mu-with-y}.
\end{proofof}

\subsection{Proofs for Section \ref{sec:generalizing}} \label{app:proofs-generalizing}

\begin{proofof}[Proposition \ref{prop:block-minors}]
  Block minors of $M(x,y)$ include the usual $2\times 2$ minors and thus if all such block minors vanish we have $B^X(x)\indep B^Y(y)$ by the discussion below Definition \ref{def:binarization}. Now suppose that $B^X(x)\indep B^Y(y)$ and let $L,L' \subset\{0,1\}^r$ and $R,R' \subset\{0,1\}^s$. Then we have that
  \begin{align*}
    &(\sum_{\substack{\ell_X\in L \\ \ell_Y \in R}} p(z)_{\ell_X\ell_Y})(\sum_{\substack{\ell_X\in L' \\ \ell_Y \in R'}} p(z)_{\ell_X'\ell_Y'}) - (\sum_{\substack{\ell_X'\in L' \\ \ell_Y \in R}} p(z)_{\ell_X'\ell_Y})(\sum_{\substack{\ell_X\in L \\ \ell_Y \in R'}} p(z)_{\ell_X\ell_Y'}) \\
    &= P(B^X(x) \in L,\ B^Y(y)\in R)\ P(B^X(x) \in L',\ B^Y(y)\in R') \\
    &\quad - P(B^X(x) \in L',\ B^Y(y)\in R)\ P(B^X(x) \in L,\ B^Y(y)\in R') \\
    &= P(B^X(x) \in L)\ P(B^Y(y)\in R)\ P(B^X(x) \in L')\ P(B^Y(y)\in R') \\
    &\quad - P(B^X(x) \in L') \ P(B^Y(y)\in R)\ P(B^X(x) \in L)\ P(B^Y(y)\in R') \\
    &= 0
  \end{align*}
  by independence.
\end{proofof}

\begin{proofof}[Proposition \ref{prop:isms-are-srcs}]
  Recall that for any $d\geq 1$, $0_d\in \bR^d$ is the vector of all $0$s. Let $d = r+s$. By definition and simple algebra
  \begin{align*}
    A(x,y)^2 &= \Big(\sum_{\substack{\ell_X\in L \\ \ell_Y \in R}}(p(z)_{0_d}p(z)_{\ell_X\ell_Y} - p(z)_{\ell_X0_s}p(z)_{0_r\ell_Y})\Big)^2 \\
             &= p(z)_{0_d}^2(\sum_{\substack{\ell_X\in L \\ \ell_Y \in R}}p(z)_{\ell_X\ell_Y})^2 + (\sum_{\ell_X\in L}p(z)_{\ell_X0_s})^2(\sum_{\ell_Y\in R}p(z)_{0_r\ell_Y})^2 \\
             &- 2\ p(z)_{0_d}(\sum_{\substack{\ell_X\in L \\ \ell_Y \in R}}p(z)_{\ell_X\ell_Y})(\sum_{\ell_X\in L}p(z)_{\ell_X0_s})(\sum_{\ell_Y\in R}p(z)_{0_r\ell_Y}).
  \end{align*}

   Also, since $\lambda_{XY}(x,y) = \prod_{i=1}^t F_{X_{E_i}Y_{F_i}}(x_{E_i},y_{F_i})$, $\lambda_{XY}$ is the cumulative distribution function of a random vector in $\bR^{r+s}$ whose entries are taken from $Z^{5,\dots,4+t}$, in particular we may, for $j\in [r+s]$, let
  \begin{align*}
    W_j = \twopartdef{X^{4 + k}}{j\in [r] \text{ and } j\in E_k,\quad \text{and}}{Y^{4+k}}{r+1\leq j\leq r+s \text{ and } j-r\in F_k}.
  \end{align*}
  Here $W$ is just a projection of $Z^{5,\dots,4+t}$. Given this fact we have that, for any integrable function $g$, $\int_{\bR^{r+s}}g(z)\dlambda_{XY} = E[g(W)]$. Now, by a direct computation, we have that
  \begin{align*}
    \int_{\bR^{r+s}} p(z)_{0_d}^2&(\sum_{\substack{\ell_X\in L \\ \ell_Y \in R}}p(z)_{\ell_X\ell_Y})^2 \dlambda_{XY}(x,y) \\
                                    &= \int_{\bR^{r+s}} E\Big[
                                      1_{[Z^1, Z^2\preceq z]}
                                      \Big(\sum_{\ell^X\in L}1_{[X^3,X^4 \lesseqgtr_{\ell^X} x]}\sum_{\ell^Y\in R}1_{[Y^3,Y^4 \lesseqgtr_{\ell^Y} y]}\Big)\dlambda_{XY}(x,y) \\
                                    &= E\Big[
                                      1_{[Z^1, Z^2\preceq W]}
                                      \Big(\sum_{\ell^X\in L}1_{[X^3,X^4 \lesseqgtr_{\ell^X} W^X]}\sum_{\ell^Y\in R}1_{[Y^3,Y^4 \lesseqgtr_{\ell^Y} W^Y]}\Big)\Big] \\
                                    & E\Big[I_X(X^{[4+t]})I_Y(Y^{[4+t]})\Big]
  \end{align*}

  Similarly we have that
  \begin{align*}
    \int_{\bR^{r+s}}& (\sum_{\ell_X\in L}p(z)_{\ell_X0_s})^2(\sum_{\ell_Y\in R}p(z)_{0_r\ell_Y})^2 \dlambda_{XY}(x,y) = E\Big[I_X(X^{[4+t]})I_Y(Y^{4,2,3,1,5,\dots,t})\Big]
  \end{align*}
  and
  \begin{align*}
    \int_{\bR^{r+s}}& p(z)_{0_d}(\sum_{\substack{\ell_X\in L \\ \ell_Y \in R}}p(z)_{\ell_X\ell_Y})(\sum_{\ell_X\in L}p(z)_{\ell_X0_s})(\sum_{\ell_Y\in R}p(z)_{0_r\ell_Y}) \dlambda_{XY}(x,y)\\
                    &= E\Big[I_X(X^{[4+t]})I_Y(Y^{1,3,2,4,5,\dots,t})\Big] \\
                    &= E\Big[I_X(X^{[4+t]})I_Y(Y^{4,2,3,1,5,\dots,t})\Big].
  \end{align*}
  Thus
  \begin{align*}
    &\int_{\bR^{r+s}}A(x,y)^2 \dlambda_{XY}(x,y) \\
                    &=E\Big[I_X(X^{[4+t]})\Big(I_Y(Y^{[4+t]}) + I_Y(Y^{4,2,3,1,5,\dots,t}) - I_Y(Y^{4,2,3,1,5,\dots,t}) - I_Y(Y^{1,3,2,4,5,\dots,t})\Big)\Big].
  \end{align*}
  Our result then follows by Lemma \ref{lem:alternate-src-forms}.
  
  It now remains to show that $D$ and $R$ are Integrated Squared Minors. But this is easy, recall from the proof of Proposition \ref{prop:standard-measures-are-srcs} that
  \begin{align*}
    &(F_{XY}(x,y) - F_X(x)F_Y(y))^2 \\
            &= F_{XY}(x,y)^2(1-F_X(x)-F_Y(y)+F_{XY}(x,y))^2 + (F_X(x)-F_{XY}(x,y))^2(F_Y(y)-F_{XY}(x,y))^2 \\
            &\quad - 2\ F_{XY}(x,y)(1-F_X(x)-F_Y(y)+F_{XY}(x,y))(F_X(x)-F_{XY}(x,y))(F_Y(y)-F_{XY}(x,y))\\
            &= c^+(x,y) + c^-(x,y) - 2d(x,y)
  \end{align*}
  where $c^+,c^-,$ and $d$ are given by Equations \eqref{eq:multivariate-D-con-p}--\eqref{eq:multivariate-D-dis}.
  But
  \begin{align*}
    F_{XY}(x,y) &= p(z)_{0_d}, \\
    F_Y(y)-F_{XY}(x,y) &= \sum_{ \ell_X\in\{0,1\}^r\setminus\{0_r\}}  p(z)_{\ell_X0_s}, \\
    F_X(x)-F_{XY}(x,y) &= \sum_{ \ell_Y\in\{0,1\}^s\setminus\{0_s\}}  p(z)_{0_r\ell_Y}, \text{ and}\\
    1-F_X(x)-F_Y(y)+F_{XY}(x,y) &= \sum_{\substack{\ell_X\in\{0,1\}^r\setminus\{0_r\} \\ \ell_Y\in\{0,1\}^s\setminus\{0_s\}}}  p(z)_{\ell_X\ell_Y}.
  \end{align*}
  Thus $(F_{XY}(x,y) - F_X(x)F_Y(y))^2$ equals the square of the $2\times 2$ block minor of $M(x,y)$ along $(\{0_r\},\ \{0,1\}^r\setminus \{0_r\},\ \{0_s\},\ \{0,1\}^s\setminus \{0_s\})$. That is, we have
  \begin{align}
    D(X,Y) &= \int_{\bR^{r+s}} \Big(\sum_{\substack{\ell_X\in\{0,1\}^r\setminus\{0\} \\ \ell_Y\in\{0,1\}^s\setminus\{0\}}} (p(z)_{0_d}p(z)_{\ell_X\ell_Y} - p(z)_{0_r\ell_Y}p(z)_{\ell_X0_s})\Big)^2 \dF_{XY}(x,y) \label{eq:D-equal-ism}\\
           &=\int_{\bR^{r+s}} A(x,y)^2\dF_{XY}(x,y) \nonumber.
  \end{align}
  Thus $D(X,Y)$ is indeed an Integrated Square Minor as claimed. That $R(X,Y)$ is an Integrated Square Minor follows in exactly the same way.
\end{proofof}

\begin{proofof}[Proposition \ref{prop:consistent-ssrcs}]
  All $\mu^{joint}_i,\mu^{prod}_i$ are Symmetric Rank Covariances by Proposition \ref{prop:isms-are-srcs}. It then follows, by definition, that $\mu^{joint}$ and $\mu^{prod}$ are Summed Symmetric Rank Covariances. We first show that $\mu^{joint} = 0 \implies D = 0$ so that $\mu^{joint}$ is consistent whenever $D$ is. Recalling \eqref{eq:D-equal-ism} and the fact that the $L_i\times R_i$ partition $(\{0,1\}^r\setminus\{0_r\}) \times (\{0,1\}^s\setminus\{0_s\})$ we have that
  \begin{align*}
    D(X,Y) &= \int_{\bR^{r+s}} \Big(\sum_{\substack{\ell_X\in\{0,1\}^r\setminus\{0_r\} \\ \ell_Y\in\{0,1\}^s\setminus\{0_s\}}} (p(z)_{0_d}p(z)_{\ell_X\ell_Y} - p(z)_{0_r\ell_Y}p(z)_{\ell_X0_s})\Big)^2 \dF_{XY}(x,y) \\
           &= \int_{\bR^{r+s}} \Big(\sum_{i=1}^k\sum_{\substack{\ell_X\in L_i \\ \ell_Y\in R_i}} (p(z)_{0_d}p(z)_{\ell_X\ell_Y} - p(z)_{0_r\ell_Y}p(z)_{\ell_X0_s})\Big)^2 \dF_{XY}(x,y).
  \end{align*}
  Now recall that if $a,b\in\bR$ then $(a + b)^2 > 0 \implies a^2 + b^2 >0$. Applying this fact $k$ times we have that
  \begin{align*}
    &\Big(\sum_{i=1}^k\sum_{\substack{\ell_X\in L_i \\ \ell_Y\in R_i}}  (p(z)_{0_d}p(z)_{\ell_X\ell_Y} - p(z)_{0_r\ell_Y}p(z)_{\ell_X0_s})\Big)^2 >0 \\
    \implies& \sum_{i=1}^k\Big(\sum_{\substack{\ell_X\in L_i \\ \ell_Y\in R_i}} (p(z)_{0_d}p(z)_{\ell_X\ell_Y} - p(z)_{0_r\ell_Y}p(z)_{\ell_X0_s})\Big)^2 >0.
  \end{align*}
  From this it immediately follows that $\mu^{joint} = 0 \implies D = 0$ as claimed. An essentially identical argument shows that $\mu^{prod} =0 \implies R = 0$. Since $R$ is D-consistent in all cases this implies that $\mu^{prod}$ is also.
\end{proofof}

\subsection{Proofs for Section \ref{sec:estimation-via-u-statistics}} \label{app:proofs-estimation-via-u-statistics}

\begin{proofof}[Proposition \ref{prop:rewrite-sym-kernel}]
  We will only show Equation \eqref{eq:rewrite-sym-kernel-in-x}, Equation \eqref{eq:rewrite-sym-kernel-in-y} then follows by symmetry. Now
  \begin{align*}
    \kappa(z^{[m]}) &= \frac{1}{m!}\sum_{\gamma\in S_m} \Big(\sum_{\sigma\in H}\sign(\sigma)\ I_X(x^{\sigma\gamma[m]})\Big) \ \Big(\sum_{\sigma\in H}\sign(\sigma)\ I_Y(y^{\sigma\gamma[m]})\Big) \\
    &= \frac{1}{m!}\sum_{\sigma\in H} \sum_{\gamma\in S_m}\sign(\sigma)\ I_X(x^{\sigma\gamma[m]}) \ \Big(\sum_{\psi\in H}\sign(\psi)\ I_Y(y^{\psi\gamma[m]})\Big) \\
    &= \frac{1}{m!}\sum_{\sigma\in H} \sum_{\sigma^{-1}\gamma\in S_m}\sign(\sigma)\ I_X(x^{\sigma\sigma^{-1}\gamma[m]}) \ \Big(\sum_{\psi\in H}\sign(\psi)\ I_Y(y^{\psi\sigma^{-1}\gamma[m]})\Big) \\
    &= \frac{1}{m!}\sum_{\sigma\in H} \sum_{\sigma^{-1}\gamma\in S_m} I_X(x^{\gamma[m]}) \ \Big(\sum_{\psi\in H}\sign(\psi)\sign(\sigma)\ I_Y(y^{\psi\sigma^{-1}\gamma[m]})\Big) \\
    &= \frac{1}{m!}\sum_{\sigma\in H} \sum_{\sigma^{-1}\gamma\in S_m} I_X(x^{\gamma[m]}) \ \Big(\sum_{\psi\in H}\sign(\psi\sigma^{-1})\ I_Y(y^{(\psi\sigma^{-1})\gamma[m]})\Big) \\
    &= \frac{1}{m!}\sum_{\sigma\in H} \sum_{\sigma^{-1}\gamma\in S_m} I_X(x^{\gamma[m]}) \ \Big(\sum_{\psi\in H}\sign(\psi)\ I_Y(y^{\psi\gamma[m]})\Big) \\
    &= \frac{|H|}{m!} \sum_{\gamma\in S_m} I_X(x^{\gamma[m]}) \ \Big(\sum_{\psi\sigma\in H}\sign(\psi) I_Y(y^{\psi\gamma[m]})\Big),
  \end{align*}
  where we have used the fact that for any $\sigma,\psi\in H$, $\sigma H=H$ and $\sigma S_m=S_m$, $\sign(\psi\sigma^{-1}) = \sign(\psi)\sign(\sigma^{-1})$, and $\sign(\sigma) = \sign(\sigma^{-1})$.
\end{proofof}

\begin{proofof}[Proposition \ref{prop:computational-results}]
  See Appendix \ref{app:efficient-computation}.
\end{proofof}

\begin{proofof}[Lemma \ref{lem:kernel-flip}]
  Without loss of generality assume that $S = \{\ell+1,...,m\}$. Moreover let $W^i = (W^i_X,W^i_Y)$ be a partition of $W^i$ into its $X$ and $Y$ components. We wish to show that $E[k(Z^{[t]}, z^{\ell+1,...,m})] = 0$. By $X\indep Y$ we have that
  \begin{align*}
    E[k(W)] = E\Big[\sum_{\sigma \in H} \sign(\sigma)\ I_{X}(W_X^{\sigma[m]})\Big]\ E\Big[\sum_{\sigma \in H}\sign(\sigma)\ I_{Y}(W_Y^{\sigma[m]})\Big].
  \end{align*}
  Now, letting $\tau_i\in E_i$ be a representative of the $i$th equivalence class, note that
  \begin{align*}
    E\Big[\sum_{\sigma \in H} \sign(\sigma)\ I_{X}(W_X^{\sigma[m]})\Big]
    &= E\Big[\sum_{i=1}^t\sum_{\sigma \in E_i} \sign(\sigma)\ I_{X}(W_X^{\sigma[m]})\Big]\\
    &= \sum_{i=1}^t\sum_{\sigma \in E_i} \sign(\sigma)\ E[I_{X}(W_X^{\sigma[m]})] \\
    &= \sum_{i=1}^t E[I_{X}(W_X^{\tau_i [m]})] \sum_{\sigma \in E_i} \sign(\sigma)\\
    &= \sum_{i=1}^t E[I_{X}(W_X^{\tau_i [m]})]\ 0 \\
    &= 0
  \end{align*}
  where the third equality follows since the $X^i$ are independent and identically distributed and so we may relabel them in $W_X^{\sigma[m]}$ so long as we preserve the locations of the $x_i$ and the fourth equality follows as each $E_i$ contains an equal number of even and odd permutations. It follows that $E[k(W)] = 0$.
\end{proofof}

\begin{proofof}[Proposition \ref{prop:multivariate-degenerate}]
  By Lemma \ref{lem:kernel-flip}, $E[k(\sigma(z^1,Z^2,\dots,Z^m))]=0$ for all $\sigma\in S_m$.  Hence,
\begin{align*}
  \kappa_1(z^1) = E[\kappa(z^1,Z^2,Z^3,\dots,Z^{m})] = \frac{1}{m!}\sum_{\sigma\in S_{m}} E[k(\sigma(z^1,Z^2,\dots,Z^m))] =0.
\end{align*}
\end{proofof}

\begin{proofof}[Lemma \ref{lem:simple-sym-kernel}]
  We first show that $\kappa_1(z^1)\equiv 0$ so that $U_{\mu}$ is degenerate. To see this let $S=\{i\}$ for some $i\in [m]$. If $i>4$ then every element of $h\in H_{\tau^*}$ fixes $i$ and hence, using lemma \ref{lem:kernel-flip} with $g=e$, we see that all $h\in H_{\tau^*}$ are in the same equivalence class. It follows that said equivalence class has an equal number of even and odd permutations and thus $E[k(Z^{1,\dots,i-1},z^i,Z^{i+1,\dots,m})] = 0$. Now suppose that $i\leq 4$. Using lemma \ref{lem:kernel-flip} it is easy to check, using the fact that the invariance group of $\mu$ contains $G$, that $H_{\tau^*}$ is divided into two equivalence classes both of which contain an even and an odd permutation and thus lemma \ref{lem:kernel-flip} holds. Since lemma \ref{lem:simple-sym-kernel} holds whenever $S$ is a singleton set, it follows, by proposition \ref{prop:multivariate-degenerate}, that $U_{\mu}$ is degenerate as claimed.

  We now show that $\kappa_2$ has the claimed form. Recall that
  \begin{align}
    \kappa_2(z^1,z^2) &= E[\kappa(z^1,z^2,Z^{3,\dots,m})] \\
                      &= \frac{1}{m!}\sum_{\sigma\in S_m} E[k(\sigma(z^1,z^2,Z^{3,\dots,m}))].
  \end{align}
  Suppose that $\sigma\in S_m$ is such that $S=\{\sigma(1),\sigma(2)\}\not\subset [4]$. Then we must have that either $\sigma(1)$ or $\sigma(2)$ is fixed by all elements in $h\in H_{\tau^*}$. It is then easy to check that, similarly as above, the conditions of lemma \ref{lem:simple-sym-kernel} hold for $S$ and thus $E[k(\gamma(z^1,z^2,Z^{3,\dots,m})] = 0$. Hence we have that
  \begin{align*}
    \kappa_2(z^1,z^2)  &= \frac{1}{m!}\sum_{\substack{\sigma\in S_m\\ \sigma(1),\sigma(2)\in [4]}} E[k(\sigma(z^1,z^2,Z^{3,\dots,m}))].
  \end{align*}
  As the $Z^i$ are independent and identically distributed and thus exchangeable it follows that, for any $\sigma,\gamma\in S_m$, if $\sigma(i)=\gamma(i)$ for $i=1,2$ then $E[k(\sigma(z^1,z^2,Z^{3,\dots,m})] = E[k(\gamma(z^1,z^2,Z^{3,\dots,m})]$. This allows us to write
  \begin{align*}
    \kappa_2(z^1,z^2) &= \frac{(m-2)!}{m!}\sum_{\substack{\sigma\in S_4\\ \sigma(3) < \sigma(4)}} E[k(\sigma(z^1,z^2,Z^3,Z^4),Z^{5,\dots,m})].
  \end{align*}
  Now
  \begin{align*}
    \sum_{\substack{\sigma\in S_4\\ \sigma(3) < \sigma(4) \\ \sigma(1)<\sigma(2)}}
 &E[k(\sigma(z^1,z^2,Z^3,Z^4),Z^{5,\dots,m})]\\
 &= E[k(z^1,z^2,Z^3,Z^4,Z^{5,\dots,m})] + E[k(z^1,Z^3,z^2,Z^4,Z^{5,\dots,m})] \\
 &\quad + E[k(z^1,Z^3,Z^4,z^2,Z^{5,\dots,m})] + E[k(Z^3,z^1,z^2,Z^4,Z^{5,\dots,m})] \\
 &\quad + E[k(Z^3,z^1,Z^4, z^2,Z^{5,\dots,m})] + E[k(Z^3,Z^4,z^1,z^2,Z^{5,\dots,m})].
  \end{align*}
  By definition of $k$ it is easy to verify, using the fact that $G$ is a subset of the invariance group, that
  \begin{align*}
    E[k(z^1,Z^3,Z^4,z^2,Z^{5,\dots,m})] = E[k(Z^3,z^1,z^2,Z^4,Z^{5,\dots,m})] = 0
  \end{align*}
  and 
  \begin{align*}
    &E[k(z^1,z^2,Z^3,Z^4,Z^{5,\dots,m})] = E[k(z^1,Z^3,z^2,Z^4,Z^{5,\dots,m})] \\
    &\quad= E[k(Z^3,z^1,Z^4, z^2,Z^{5,\dots,m})] = E[k(Z^3,Z^4,z^1,z^2,Z^{5,\dots,m})].
  \end{align*}
  Thus
  \begin{align*}
    \sum_{\substack{\sigma\in S_4\\ \sigma(3) < \sigma(4) \\ \sigma(1)<\sigma(2)}}
    E[k(\sigma(z^1,z^2,Z^3,Z^4),Z^{5,\dots,m})] = 4E[k(z^1,z^2,Z^3,Z^4,Z^{5,\dots,m})].
  \end{align*}
  By symmetry and as one may easily check that
  \begin{align*}
    E[k(z^1,z^2,Z^3,Z^4,Z^{5,\dots,m})]=E[k(z^2,z^1,Z^3,Z^4,Z^{5,\dots,m})]
  \end{align*}
  we have
  \begin{align*}
    \kappa_2(z^1,z^2) &= \frac{(m-2)!}{m!}\sum_{\substack{\sigma\in S_4\\ \sigma(3) < \sigma(4)}} E[k(\sigma(z^1,z^2,Z^3,Z^4),Z^{5,\dots,m})]\\
                      &= \frac{(m-2)!}{m!}\Big(\sum_{\substack{\sigma\in S_4\\ \sigma(3) < \sigma(4)\\ \sigma(1)<\sigma(2)}} E[k(\sigma(z^1,z^2,Z^3,Z^4),Z^{5,\dots,m})] \\
                      &\hspace{20mm} + \sum_{\substack{\sigma\in S_4\\ \sigma(3) < \sigma(4) \\ \sigma(2) < \sigma(1)}} E[k(\sigma(z^1,z^2,Z^3,Z^4),Z^{5,\dots,m})]\Big)\\
                      &= \frac{(m-2)!}{m!}\Big(4E[k(z^1,z^2,Z^3,Z^4,Z^{5,\dots,m})] + 4 E[k(z^2,z^1,Z^3,Z^4,Z^{5,\dots,m})]\Big)\\
                      &= \frac{8 (m-2)!}{m!}E[k(z^1,z^2,Z^3,Z^4,Z^{5,\dots,m})].
  \end{align*}
  Rewriting $\frac{8 (m-2)!}{m!} = \frac{4}{{m\choose 2}}$ gives our claimed result.
\end{proofof}

\begin{proofof}[Proposition \ref{prop:asymptotic-dists-for-bivariate}]
  Let $\kappa^{\tau^*},\kappa^D,\kappa^R$ be the symmetrized kernels corresponding to $\tau^*,D$ and $R$ respectively. By determining an explicit representation of $\kappa^{\tau^*}_2(z^1,z^2)$, Theorem 4.4 in \cite{NandyEtAl16} shows that $n\ U_{\tau^*} \to \frac{36}{\pi^4} Z$ in distribution. Indeed, they show that the eigenvalues associated to $\kappa^{\tau^*}_{2}$ are $\lambda_{ij}=\frac{6}{\pi^4i^2j^2}$ for $i,j\in \bZ_{>0}$. We will now show that $\kappa^D_2$ and $\kappa^R_2$ are scalar multiples of $\kappa^{\tau^*}_2$. By Lemma \ref{lem:simple-sym-kernel} we have that
  \begin{align*}
    \kappa^{\tau^*}_2(z^1,z^2) = \frac{2}{3}\ E[a_{I_{\tau^*}}(x^1,x^2,X^3,X^4)]\ E[a_{I_{\tau^*}}(y^1,y^2,Y^3,Y^4)].
  \end{align*}
  By Lemma \ref{lem:simple-sym-kernel} and Proposition \ref{prop:standard-measures-are-srcs} we have that
  \begin{align*}
    \kappa^D_2&(z^1,z^2) = \frac{1}{10}\ E\Big[a_{I_D}(x^{1,2},X^{3,4,5})\Big]\ E\Big[a_{I_D}(y^{1,2},Y^{3,4,5})\Big].
  \end{align*}
  Now a lengthy but straightforward computation shows that
  \begin{align*}
    E[a_{I_D}(x^{1,2},X^{3,4,5})] 
      &= \frac{1}{3} E[a_{I_{\tau^*}}(x^{1,2},X^{3,4})]
  \end{align*}
  and similarly for $E[a_{I_D}(y^{1,2},Y^{3,4,5})]$. Thus
  \begin{align*}
    \kappa^D_2&(z^1,z^2) \;=\; \frac{1}{90}\ E[a_{I_{\tau^*}}(x^{1,2},X^{3,4})]\ E[a_{I_{\tau^*}}(y^{1,2},Y^{3,4})] 
                \;=\; \frac{1}{60}\ \kappa^{\tau^*}_2(z^1,z^2).
  \end{align*}
  An essentially identical computations shows $\kappa^R_2(z^1,z^2) = \frac{1}{90} \kappa^{\tau^*}_2(z^1,z^2)$. Now, given knowledge of the eigenvalues of $\kappa^{\tau^*}_2$ above, our results follow immediately from Theorem \ref{theorem: degenerate asymptotics}.
\end{proofof}

\section{Efficient Computation of $U_D,U_R,U_{\tau^*_P}$, and $U_{\tau^*_J}$} \label{app:efficient-computation}

\subsection{Computing $U_D$} \label{sec:computing-ud}

By Proposition \ref{prop:rewrite-sym-kernel} we have that
\begin{align*}
  &\kappa^D(z^{[5]}) = \frac{1}{5!}\sum_{\gamma\in S_m}\frac{1}{4} \Big(\sum_{\sigma\in H_{\tau^*}}\sign(\sigma)I_{D}(x^{\sigma\gamma[5]})\Big)\Big(\sum_{\sigma\in H_{\tau^*}}\sign(\sigma)I_{D}(y^{\sigma\gamma[5]})\Big) \\
  &= \frac{1}{5!}\sum_{\gamma\in S_m}\frac{|H_{\tau^*}|}{4} I_{D}(x^{\gamma[5]})\Big(\sum_{\sigma\in H_{\tau^*}}\sign(\sigma)I_{D}(y^{\sigma\gamma[5]})\Big)\\
  &= \frac{1}{5!}\sum_{\gamma\in S_m} I_{D}(x^{\gamma[5]})\Big(\sum_{\sigma\in H_{\tau^*}}\sign(\sigma)I_{D}(y^{\sigma\gamma[5]})\Big).
\end{align*}
Recalling the definition of $I_D$ this gives
\begin{align*}
  U_{D}(z^{[n]}) = \frac{1}{{n\choose 5}5!}&\sum_{\substack{1\leq i_1,\dots,i_5\leq n\\ i_1\not=i_2\not=\dots\not=i_5}} I_{[x^{i_1},x^{i_2}\preceq x^{i_5}]}I_{[x^{i_3},x^{i_4}\not\preceq x^{i_5}]}\\
                                           &\cdot \Big(I_{[y^{i_1},y^{i_2}\preceq y^{i_5}]}I_{[y^{i_3},y^{i_4}\not\preceq y^{i_5}]} + I_{[y^{i_4},y^{i_3}\preceq y^{i_5}]}I_{[y^{i_2},y^{i_1}\not\preceq y^{i_5}]} \\
                                           &\quad - I_{[y^{i_1},y^{i_3}\preceq y^{i_5}]}I_{[y^{i_2},y^{i_4}\not\preceq y^{i_5}]} - I_{[y^{i_4},y^{i_2}\preceq y^{i_5}]}I_{[y^{i_3},y^{i_1}\not\preceq y^{i_5}]}\Big).
\end{align*}
Now for any $1\leq k\leq n$ define
\begin{align*}
  C_{\preceq,\preceq}(k) :&= |\{i\ :\ i\not=k\text{ and } x^i\preceq x^{k} \text{ and } y^i \preceq y^k\}| \\
                     &= |\{i\ :\ z^i\preceq z^{k}\}| - 1 \\
  C_{\preceq,\not\preceq}(k) :&= |\{i\ :\ i\not=k\text{ and } x^i\preceq x^{k} \text{ and } y^i \not\preceq y^k\}| \\
                     &= |\{i\ :\ x^i\preceq x^{k}\}| - |\{i\ :\ z^i\preceq z^{k}\}|, \\
  C_{\not\preceq,\preceq}(k) :&= |\{i\ :\ i\not=k\text{ and } x^i\not\preceq x^{k} \text{ and } y^i \preceq y^k\}| \\
                      &=|\{i\ :\ y^i\preceq y^{k}\}| - |\{i\ :\ z^i\preceq z^{k}\}|, \quad\quad\quad \text{and} \\
  C_{\not\preceq,\not\preceq}(k) :&= |\{i\ :\ i\not=k\text{ and } x^i\not\preceq x^{k} \text{ and } y^i \not\preceq y^k\}| \\
  &=n - |\{i\ :\ x^i\preceq x^{k}\}| - |\{i\ :\ y^i\preceq y^{k}\}| + |\{i\ :\ z^i\preceq z^{k}\}|.
\end{align*}
From this, for fixed $i_5$,
\begin{align*}
  &\sum_{\substack{1\leq i_1,\dots,i_4\leq n\\ i_1\not=i_2\not=\dots\not=i_5}} I_{[x^{i_1},x^{i_2}\preceq x^{i_5}]}I_{[x^{i_3},x^{i_4}\not\preceq x^{i_5}]}I_{[y^{i_1},y^{i_2}\preceq y^{i_5}]}I_{[y^{i_3},y^{i_4}\not\preceq y^{i_5}]} \\
  &= |\{\text{pairs $i\not=j$ with $i\not=i_5\not=j$ and $z^i,z^j\preceq z^{i_5}$}\}|\ |\{\text{pairs $i,j$ with $x^i,x^j\not\preceq x^{i_5}$ and $y^i,y^j\not\preceq y^{i_5}$}\}| \\
  &= 4\ {C_{\preceq,\preceq}(i_5) \choose 2 } \ {C_{\not\preceq,\not\preceq}(i_5) \choose 2 } \\
  &= A(i_5).
\end{align*}
Similarly we have
\begin{align*}
  &\sum_{\substack{1\leq i_1,\dots,i_4\leq n\\ i_1\not=i_2\not=\dots\not=i_5}} I_{[x^{i_1},x^{i_2}\preceq x^{i_5}]}I_{[x^{i_3},x^{i_4}\not\preceq x^{i_5}]} I_{[y^{i_4},y^{i_3}\preceq y^{i_5}]}I_{[y^{i_2},y^{i_1}\not\preceq y^{i_5}]} \\
  &= |\{\text{pairs $i,j$ with $x^i,x^j\preceq x^{i_5}$ and $y^i,y^j\not\preceq y^{i_5}$}\}|\ |\{\text{pairs $i,j$ with $x^i,x^j\not\preceq x^{i_5}$ and $y^i,y^j \preceq y^{i_5}$}\}| \\
  &= 4\ {C_{\preceq,\not\preceq}(i_5)\choose 2 } \ {C_{\preceq,\not\preceq}(i_5)\choose 2 }\\
  &=B(i_5),
\end{align*}
and
\begin{align*}
  &\sum_{\substack{1\leq i_1,\dots,i_4\leq n\\ i_1\not=i_2\not=\dots\not=i_5}} I_{[x^{i_1},x^{i_2}\preceq x^{i_5}]}I_{[x^{i_3},x^{i_4}\not\preceq x^{i_5}]} I_{[y^{i_1},y^{i_3}\preceq y^{i_5}]}I_{[y^{i_2},y^{i_4}\not\preceq y^{i_5}]} \\
  &=\sum_{\substack{1\leq i_1,\dots,i_4\leq n\\ i_1\not=i_2\not=\dots\not=i_5}} I_{[x^{i_1},x^{i_2}\preceq x^{i_5}]}I_{[x^{i_3},x^{i_4}\not\preceq x^{i_5}]} I_{[y^{i_4},y^{i_2}\preceq y^{i_5}]}I_{[y^{i_3},y^{i_1}\not\preceq y^{i_5}]} \\
  &= C_{\preceq,\preceq}(i_5)\ C_{\not\preceq,\not\preceq}(i_5)\ C_{\not\preceq,\preceq}(i_5)\ C_{\preceq,\not\preceq}(i_5) \\
  &=C(i_5).
\end{align*}
Thus we have that
\begin{align} \label{eq:D-u-stat-for-computation}
  U_{D}(z^1,\dots,z^n) = \frac{1}{{n\choose 5}5!}\sum_{1\leq i\leq n} (A(i) + B(i) - 2 \ C(i)).
\end{align}
Now it is easy to verify that, for any $1\leq i\leq n$, we may compute $A(i),B(i),$ and $C(i)$ using a constant number of orthogonal range queries on $z^{[n]}$. Noting that it takes $O(n\ \log_2(n)^{d-1})$ to construct the range-tree on $z^{[n]}$, each orthogonal range query takes $O(\log_2(n)^{d-1})$ time, and there are $n$ iterations in the above sum, it follows that we may compute $U_d(z^1,\dots,z^n)$ in $O(n\ \log_2(n)^{d-1}) + n\ O(\log_2(n)^{d-1}) = O(n\ \log_2(n)^{d-1})$ time.

\subsection{Computing $U_R$}

Recall that the kernel $\kappa^R$ is of order $m = d+4$ and so naively takes $O(n^{d+4})$ time to compute, we will show that it can be computed in $O(n^d)$ time. By similar arguments as in Section \ref{sec:computing-ud} we have that
\begin{align} \label{eq:R-u-stat-for-computation}
  U_R(z^{[n]}) = \frac{1}{{n\choose m}m!}\sum_{\substack{i^{[d]} = (i^1,\dots,i^d) \in [n]^d\\ i^1\not=\dots\not= i^d}}(A^R(i^{[d]})+B^R(i^{[d]})-2\ C^R(i^{[d]}))
\end{align}
where
\begin{align*}
  A^R(i^{[d]}) &= 4\ {C_{\preceq,\preceq}^R(i^{[d]}) \choose 2} {C_{\not\preceq,\not\preceq}^R(i^{[d]})\choose 2}, \\
  B^R(i^{[d]}) &= 4\ {C_{\preceq,\not\preceq}^R(i^{[d]}) \choose 2} {C_{\not\preceq,\preceq}^R(i^{[d]})\choose 2},\quad \text{and} \\
  C^R(i^{[d]}) &= C_{\preceq,\preceq}^R(i^{[d]})\ C_{\not\preceq,\not\preceq}^R(i^{[d]})\ C_{\preceq,\not\preceq}^R(i^{[d]})\ C_{\not\preceq,\preceq}^R(i^{[d]})
\end{align*}
and, letting $w=(w^X,w^Y)\in\bR^{r+s}$ be such that $w^X_j = x^{i^j}_j$ for $j\in [r]$ and $w^Y_j = y^{i^j}_j$ for $j\in[s]$,
\begin{align*}
  C^R_{\preceq,\preceq}(i^{[d]}) :&= |\{i\ :\ i\not\in\{i^1,\dots,i^d\}\text{ and } x^i\preceq w^X\ \text{ and } y^i \preceq w^Y\}| \\
                     &= |\{i\ :\ z^i\preceq w\}| - |\{j\ :\ z^{i^j}\preceq w\}| \\
  C^R_{\preceq,\not\preceq}(i^{[d]}) :&= |\{i\ :\ i\not\in\{i^1,\dots,i^d\} \text{ and } x^i\preceq w^X \text{ and } y^i \not\preceq w^Y\}| \\
                     &= |\{i\ :\ x^i\preceq w^X\}| - |\{i\ :\ z^i\preceq w\}| - |\{j\ :\ x^{i^j}\preceq w^X\text{ and }y^{i^j}\not\preceq w^Y\}|, \\
  C^R_{\not\preceq,\preceq}(i^{[d]}) :&= |\{i\ :\ i\not\in\{i^1,\dots,i^d\} \text{ and } x^i\not\preceq w^X \text{ and } y^i \preceq W^Y\}| \\
                                  &= |\{i\ :\ y^i\preceq w^Y\}| - |\{i\ :\ z^i\preceq w\}| - |\{j\ :\ x^{i^j}\not\preceq w^X\text{ and }y^{i^j}\preceq w^Y\}|, \\
  C^R_{\not\preceq,\not\preceq}(i^{[d]}) :&= |\{i\ :\ i\not\in\{i^1,\dots,i^d\}\text{ and } x^i\not\preceq w^X \text{ and } y^i \not\preceq w^Y\}| \\
                                  &=n - |\{i\ :\ x^i\preceq w^X\}| - |\{i\ :\ y^i\preceq w^Y\}| + |\{i\ :\ z^i\preceq w\}|\\
  &\quad\quad - |\{j\ :\ x^{i^j}\not\preceq w^X\text{ and } y^{i^j}\not\preceq w^Y\}|.
\end{align*}
Clearly each of $A^R(i^{[d]}),B^R(i^{[d]})$, and $C^R(i^{[d]})$ can be computed using a constant number of orthogonal range searches. Now, constructing the tensor from Proposition \ref{prop:tensor-orth-search} for $z^{[n]}$ takes $O(n^d)$ time after which orthogonal range searches on $z^{[n]}$ can be completed in constant time. The summation in Equation \eqref{eq:R-u-stat-for-computation} is over $n^d$ elements and thus, using the tensor, the summation can be completed in $O(n^d)$ time. It then follows that the total time to compute $U_R(z^{[n]})$ is $O(n^d)+O(n^d) = O(n^d)$ as claimed.

\subsection{Computing $U_{\tau^*_J}$ and $U_{\tau^*_P}$}

The computation of the U-statistics estimating $\tau^*_J$ and $\tau^*_P$ is somewhat more involved than that for those estimating $U_D$ and $U_R$. By Proposition \ref{prop:rewrite-sym-kernel} we have that
\begin{align*}
  \kappa^{\tau_P^*}(z^{[4]}) &= \frac{1}{3!}\sum_{\gamma\in S_4} I_{P}(x^{\gamma[4]})\sum_{\sigma\in H_{\tau^*}}\sign(\sigma)I_{J}(y^{\sigma\gamma[4]}).
\end{align*}
This then gives us that
\begin{align*}
  &U_{\tau^*_J}(z^{n]}) = \frac{1}{{n \choose 4} 3!}\sum_{\substack{1\leq i_1,\dots,i_4\leq n \\ i_1\not=\dots\not= i_4}}I_{[x^{i_3},x^{i_4}\not\preceq x^{i_1},x^{i_2}]} \ \Big(I_{[y^{i_3},y^{i_4}\not\preceq y^{i_1},y^{i_2}]} + I_{[y^{i_2},y^{i_1}\not\preceq y^{i_4},y^{i_3}]} \\
                       &\hspace{85mm} - I_{[y^{i_2},y^{i_4}\not\preceq y^{i_1},y^{i_3}]} - I_{[y^{i_3},y^{i_1}\not\preceq y^{i_4},y^{i_2}]}\Big) \\
  &= \frac{1}{{n \choose 4} 3!}\sum_{\substack{1\leq i_1,\dots,i_4\leq n \\ i_1\not=\dots\not= i_4}}I_{[x^{i_3},x^{i_4}\not\preceq x^{i_1},x^{i_2}]} \ \Big(I_{[y^{i_3},y^{i_4}\not\preceq y^{i_1},y^{i_2}]} + I_{[y^{i_2},y^{i_1}\not\preceq y^{i_4},y^{i_3}]}  - 2\ I_{[y^{i_2},y^{i_4}\not\preceq y^{i_1},y^{i_3}]}\Big)
\end{align*}
where the second equality follows by swapping the labels of 1,2 and 3,4 respectively. Similarly as in the prior sections, for fixed $i_1,i_2$,
\begin{align}
  &\sum_{\substack{1\leq i_3,i_4\leq n \\ i_1\not=\dots\not= i_4}}I_{[x^{i_3},x^{i_4}\not\preceq x^{i_1},x^{i_2}]} \ I_{[y^{i_3},y^{i_4}\not\preceq y^{i_1},y^{i_2}]} \nonumber\\
  &= |\{\text{pairs $k\not=l$ with $k,l\not\in \{i_1,i_2\}$, $x^{k},x^{l}\not\preceq x^{i_1},x^{i_2}$ and $y^{k},y^{l} \not\preceq y^{i_1},y^{i_2}$}\}| \nonumber\\
  &= 2\ { n - |\{i : x^i \preceq x^{i_1}\quad \text{or}\quad x^i \preceq x^{i_2}\quad \text{or}\quad y^i \preceq y^{i_1}\quad \text{or}\quad y^i \preceq y^{i_2}\}| \choose 2}. \label{eq:pos-con-tau-p}
\end{align}
Now, using the standard inclusion-exclusion formulas we may compute $|\{i : x^i \preceq x^{i_1}\quad \text{or}\quad x^i \preceq x^{i_2}\quad \text{or}\quad y^i \preceq y^{i_1}\quad \text{or}\quad y^i \preceq y^{i_2}\}|$ using 16 orthogonal range queries queries on $z^{[n]}$.

Next
\begin{align}
  &\sum_{\substack{1\leq i_3,i_4\leq n \\ i_1\not=\dots\not= i_4}}I_{[x^{i_3},x^{i_4}\not\preceq x^{i_1},x^{i_2}]} \ I_{[y^{i_2},y^{i_1}\not\preceq y^{i_4},y^{i_3}]}\nonumber \\
  &= |\{\text{pairs $k\not=l$ with $k,l\not\in \{i_1,i_2\}$, $x^{k},x^{l} \not\preceq x^{i_1},x^{i_1}$, and $y^{i_1},y^{i_2} \not\preceq y^{k},y^{l}$}\}| \nonumber \\
  &= 2\ { n - |\{i : x^i \preceq x^{i_1}\quad \text{or}\quad x^i \preceq x^{i_2}\quad \text{or}\quad y^{i_1} \preceq y^i\quad \text{or}\quad y^{i_2} \preceq y^i\}| \choose 2}. \label{eq:neg-con-tau-p}
\end{align}
Again, by inclusion-exclusion, we have that $|\{i : x^i \preceq x^{i_1}\quad \text{or}\quad x^i \preceq x^{i_2}\quad \text{or}\quad y^{i_1} \preceq y^i\quad \text{or}\quad y^{i_2} \preceq y^i\}|$ can be computed using 16 orthogonal range queries on $z^{[n]}$. We now have the most difficult case remaining. We have that
\begin{align}
  &\sum_{\substack{1\leq i_3,i_4\leq n \\ i_1\not=\dots\not= i_4}}I_{[x^{i_3},x^{i_4}\not\preceq x^{i_1},x^{i_2}]} \ I_{[y^{i_2},y^{i_4}\not\preceq y^{i_1},y^{i_3}]} \nonumber\\
  &= |\{\text{pairs $k\not=l$ with $k,l\not\in \{i_1,i_2\}$, $x^{k},x^{l}\not\preceq x^{i_1},x^{i_2}$ and $y^{i_2},y^{l} \not\preceq y^{i_1},y^{l}$}\}| \nonumber\\
  &= 1_{[y^{i_2} \not\preceq y^{i_1}]}\Bigg(2\ {n \choose 2} - |\{(k,l): x^k\preceq x^{i_1} \ \ \text{or} \ \  x^k\preceq x^{i_2} \ \ \text{or} \ \  x^l\preceq x^{i_1} \ \ \text{or} \ \  x^l\preceq x^{i_2} \label{eq:dis-tau-p}\\
  &\hspace{60mm}  \ \ \text{or} \ \ y^{k} \preceq y^l \ \ \text{or} \ \  y^{i_2} \preceq y^{l} \ \ \text{or} \ \  y^{k} \preceq y^{i_1}\}|\Bigg). \nonumber
\end{align}
Unlike in the prior derivations, the second to last term above has a condition which directly relates $z^k$ and $z^l$ and hence we cannot reduce to simple forms as in Equations \eqref{eq:pos-con-tau-p},~\eqref{eq:neg-con-tau-p}. Despite this, as we will now show, it is still possible to use orthogonal range queries to compute Equation \eqref{eq:dis-tau-p}. To see this, we construct a collection $\cD_{pair}$ of points in $\bR^{2d}$ consisting of the concatenation of all pairs $z^i,z^j$ with $i\not=j$, that is we let
\begin{align*}
  \cD_{pair} = \{(z^i,z^j) = (x^i,y^i,x^j,y^j) \mid 1\leq i\not=j\leq n\}.
\end{align*}
$|\cD_{pair}| = 2\ {n\choose 2} = n(n-1)$ so that $\cD_{pair}$ takes $O(n^2)$ time to construct. Since $\cD_{pair}$ contains $n(n-1)$ elements of dimension $2d$ we may construct a orthogonal range-tree on $\cD_{pair}$ in, recalling that we consider $d$ to be bounded, $O(n^2 \log_2(n^2)^{2d-1}) = O(n^2\log_2(n)^{2d-1})$ time. Orthogonal range queries on $\cD_{pair}$ require $O(\log_2(n)^{2d-1})$ time.

Let
\begin{align*}
  a_1 &= \{(k,l)\in B : x^k\preceq x^{i_1}\}, \quad a_2 = \{(k,l)\in B : x^k\preceq x^{i_2}\}, \\
  a_3 &= \{(k,l)\in B : x^l\preceq x^{i_1}\}, \quad a_4 = \{(k,l)\in B : x^l\preceq x^{i_2}\},\\
  a_5 &= \{(k,l)\in B : y^k\preceq y^{i_1}\}, \quad a_6 = \{(k,l)\in B : y^{i_2}\preceq y^{l}\}, \text{ and} \\
  a_7 &= \{(k,l)\in B : y^k\preceq y^{l}\}
\end{align*}
where $B = \{(k,l)\in [n]^2 : k\not=l\}$.

Using inclusion-exclusion we have that
\begin{align*}
  &2\ {n \choose 2} - |\{(k,l)\in B: x^k\preceq x^{i_1} \ \ \text{or}\ \  x^k\preceq x^{i_2} \ \ \text{or}\ \  x^l\preceq x^{i_1} \ \ \text{or}\ \  x^l\preceq x^{i_2} \ \ \\
  &\hspace{20mm}\text{or}\ \  y^{k} \preceq y^l \ \ \text{or}\ \  y^{i_2} \preceq y^{l} \ \ \text{or}\ \  y^{k} \preceq y^{i_1}\}|  \\
  &= \sum_{L\subset[7]}(-1)^{|L|} |\cap_{i\in L} a_i| \\
  &= \sum_{L\subset[6]}(-1)^{|L|} |\cap_{i\in L} a_i| + \sum_{L\subset[6]}(-1)^{|L|+1} |a_7 \cap (\cap_{i\in L} a_i)|
\end{align*}
where we, in the above, let the empty intersection equal $B$. Now $a_1,\dots,a_6$ are nothing more than orthogonal range constraints on elements in $\cD_{pair}$, it follows that $|\cap_{i\in L} a_i|$ for $L\subset [6]$ can be computed by an orthogonal range query on $\cD_{pair}$. Thus we can compute $\sum_{L\subset[6]}(-1)^{|L|} |\cap_{i\in L} a_i|$ using $2^6 = 64$ orthogonal range queries on $\cD_{pair}$ which takes $O(64\log_2(n)^{2d-1})=O(\log_2(n)^{2d-1})$ time.

It remains to show how we can compute $\sum_{L\subset[6]}(-1)^{|L|} |a_7 \cap (\cap_{i\in L} a_i)|$. Since $a_7$ describes relationship between pairs $y^k,y^l$ and thus is not a standard orthogonal range query constraint. Perhaps surprisingly, however, this does not pose a substantial obstacle. Consider the collection
\begin{align*}
  \cD^*_{pair} = \{(z^i,z^j) = (x^i,y^i,x^j,y^j) \mid 1\leq i\not=j\leq n \quad\text{and}\quad y^i \preceq y^j\},
\end{align*}
$\cD^*_{pair}$ is the subset of points in $\cD_{pair}$ which satisfy the condition in $a_7$. Clearly $D^*_{pair}$ can be constructed in $O(n^2)$ time. Hence for any $L\subset [6]$ we can compute $a_7 \cap (\cap_{i\in L} a_i)$ by performing an orthogonal range query, with the constraints from $\cap_{i\in L} a_i$, on the set $\cD^*_{pair}$. It follows that $\sum_{L\subset[6]}(-1)^{|L|} |a_7 \cap (\cap_{i\in L} a_i)|$ can be computed with $64$ orthogonal range queries, as with $\cD_{pair}$, constructing a rangetree on $\cD^*_{pair}$ takes $O(n^2\log_2(n)^{2d-1})$ and range query requires $O(\log_2(n)^{2d-1})$ time.

Finally, as $U_{\tau^*_p}(z^{[n]})$ is a sum over $n(n-1)$ choices for $i_1,i_2$ and, for each $i_1,i_2$, we must compute \eqref{eq:pos-con-tau-p}, \eqref{eq:neg-con-tau-p}, and \eqref{eq:dis-tau-p} which, from the above, requires $O(\log_2(n)^{2d-1})$ time. Hence assuming the range-trees on $\cD_{pair}$ and $\cD^*_{pair}$ have already been constructed, computing $U_{\tau^*_P}(z^{[n]})$ requires $O(n(n-1)\log_2(n)^{2d-1}) = O(n^2\log_2(n)^{2d-1})$ time. As constructing the range-trees on $\cD_{pair}$ and $\cD^*_{pair}$ require each $O(n^2\log_2(n)^{2d-1})$ time it follows that the total asymptotic computation time of $U_{\tau^*_P}(z^{[n]})$ is $O(n^2\log_2(n)^{2d-1})$.

A similar argument shows that $U_{\tau^*_J}$ can also be computed in $O(n^2\log_2(n)^{2d-1})$ time.

\bibliographystyle{abbrvnat}
\bibliography{bersma_bib}

\end{document}